\documentclass[10pt]{amsart}
\usepackage{amssymb}
\usepackage{graphicx}

\theoremstyle{plain}
\newtheorem{thm}{Theorem}
\newtheorem{prop}[thm]{Proposition}
\newtheorem{lem}[thm]{Lemma}

\newtheorem{cor}[thm]{Corollary}

\newtheorem*{claim}{Claim}
\theoremstyle{definition}
\newtheorem{problem}[thm]{Problem}

\newtheorem{exm}[thm]{Example}
\newtheorem{rem}[thm]{Remark}
\newtheorem{df}[thm]{Definition}

\newtheorem{conj}[thm]{Conjecture}

\newcommand{\jsd}{join-sem\-i\-dis\-trib\-u\-tive}

\renewcommand{\leq}{\leqslant}

\newcommand{\<}{\langle}
\renewcommand{\>}{\rangle}
\newcommand{\rto}{\rightarrow}

\newcommand{\Ji}{\operatorname{Ji}}

\newcommand{\op}{\operatorname}

\DeclareMathOperator{\Id}{Id}

\begin{document}

\title{On implicational bases of closure systems with unique critical sets}
\author {K. Adaricheva}
\address{Department of Mathematical Sciences, Yeshiva University,
245 Lexington ave.,
New York, NY 10016, USA}
\email{adariche@yu.edu}

\author {J. B. Nation}
\address{Department of Mathematics, University of Hawaii, Honolulu, HI
96822, USA}
\email{jb@math.hawaii.edu}

\thanks{The first author was partially supported by AWM-NSF Mentor Travel grant N0839954.}

\begin{abstract}
We show that every optimum basis of a finite closure system, in D. Maier's sense, is also right-side optimum, which is a parameter of a minimum CNF representation of a Horn Boolean function. New parameters for the size of the binary part are also established.
We introduce the $K$-basis of a general closure system, which is a refinement of the canonical basis of V.~Duquenne and J.L.~Guigues, and discuss a polynomial algorithm to obtain it. We study closure systems with unique critical sets, and some subclasses of these where the $K$-basis is unique. A further refinement in the form of the $E$-basis is possible for closure systems without $D$-cycles. There is a polynomial algorithm to recognize the $D$-relation from a $K$-basis.  Thus, closure systems without $D$-cycles can be effectively recognized. While the $E$-basis achieves an optimum in one of its parts, the optimization of the others is an NP-complete problem.
\end{abstract}

\maketitle

\section{Introduction}

Closure system on a finite set is a unifying concept in logic programming, 
relational data bases and knowledge systems. Closure systems can be defined by a set of implications (a basis), and in such form they appear as Horn formulas in logic programming, dependencies in relational data bases, CNF representations of Horn Boolean functions and directed hypergraphs in discrete optimization.

Closure systems can also be presented in 
the terms of finite lattices, and the tools of economic description 
of a finite lattice have long existed in lattice theory. In this paper we continue the study of economic representation of a closure system based on the structure of its closure lattice, initiated in K.~Adaricheva, J.B.~Nation and R.~Rand \cite{ANR11}.

Since the seminal work of D. Maier \cite{Mai}, the main parameters of effective representation of a closure system have been (1) the number of implications in a basis, or (2) the total number of literals in all implications of the basis. It was shown by Maier, in a non-trivial argument, that the set of implications that achieves the minimum in the second parameter (an optimum basis) also achieves the minimum in the first (a minimum basis).
The result of V.~Duquenne and J.L.~Guigues  \cite{DG} is that every closure system has a canonical minimum basis; moreover, it can be obtained from any given basis in time polynomial in the size of that basis, see A. Day \cite{D92}. Quite to the opposite, Maier showed that the problem of finding an optimum basis is NP-complete \cite{Mai}.

On the other hand, in the theory of Horn Boolean functions and directed hyper-graphs, other parameters were developed, and no connection with Maier's parameters was so far realized. We establish such a connection in Theorem \ref{main}, for the minimum representation of a Horn Boolean function.

Then we introduce a $K$-basis in general closure systems following the idea of the minimal join representation of elements in finite lattices. This produces a refinement of the canonical basis of V. Duquenne and J.L. Guigues \cite{DG}.
While the $K$-basis is not optimum, it does provide a reduction in size of the canonical basis and can be obtained in polynomial time from the canonical basis.

The $K$-basis allows us to establish an important link between the canonical basis and the $D$-relation in the closure lattice of a closure system: the latter plays an important role in the lattice theoretical literature, see R.~Freese, J.~Je\v{z}ek and J.B.~Nation \cite{FJN}. In particular, the $D$-relation can be effectively recovered from the canonical basis via its refinement to any $K$-basis.
This allows us to recognize closure systems without $D$-cycles (Theorem \ref{mainD}), which are generalizations of the quasi-acyclic systems defined in P.~Hammer and A.~Kogan\cite{HK}.

We also suggest the general concept of \emph{partial optimizations} based on the idea that any basis can always be divided into two natural parts: so-called \emph{binary} and \emph{non-binary}. Either of those parts can be optimized given various assumptions. Alternatively, a basis can be minimized with respect to the total size of all the \emph{premises} of implications, which we call \emph{left-side optimum}, or with respect to the total size of all the \emph{conclusions}, which is called \emph{right-side optimum}. 

The essential part of the paper is devoted to the study of effective representations of closure systems with unique critical sets, or $UC$-\emph{systems}. The definition of this class is based on the notion of \emph{essential} and \emph{critical} sets associated with a given closure system. In one important subclass of such systems, the $K$-basis is unique. 

Further refinement in the form of the $E$-basis is possible in systems without $D$-cycles, which form a proper subclass of $UC$-systems. This basis is  right-side optimized in its non-binary part. Still, we show that finding an optimum basis for such systems is an NP-complete problem.

The paper is organized as follows. We collect all the required definitions and recall important results in section \ref{prem}. Then section \ref{min-opt} establishes the relationship between optimum and right-side optimum bases (Theorem \ref{main}). Section \ref{bin}  deals with the problem of optimization of the binary part. It turns out, such optimization is independent of the form and size of the non-binary part (Theorem \ref{rs-bin}). We also introduce and discuss the concept of a \emph{regular} basis (Definition \ref{regD}) that becomes essential in section \ref{D-rel}.
In section \ref{kbas} we introduce the $K$-basis (Definition \ref{Kbs}) and describe a polynomial algorithm for retrieving it from the canonical basis (Proposition \ref{min gen}). This will be used in section \ref{Kbas SD}, which discusses the $K$-basis in systems with unique critical sets. In section \ref{D-rel}, a relationship between a canonical basis and the $D$-relation on the closure lattice is established (Theorem \ref{tr}), which allows us to investigate, in sections \ref{Ebas} and \ref{optE}, closure systems without $D$-cycles. The latter form a proper subclass of the class of $UC$-systems (Definition \ref{UCsys}). In particular, we prove that the $E$-basis that was introduced in \cite{ANR11} for closure systems without $D$-cycles, is optimized in one of three essential parts of the basis (Theorem \ref{rs-min}). Nevertheless, two other parts of this basis cannot be effectively optimized, which is shown in section \ref{NP} (Corollaries \ref{NP1} and \ref{sRNP}). 

\section{Preliminaries}\label{prem}

Given a non-empty set $S$ and the set $\mathbf{2}^S$ of all its subsets, a \emph{closure operator} is a map $\phi: \mathbf{2}^S \rightarrow \mathbf{2}^S$ that satisfies the following, for all $X,Y \in \mathbf{2}^S$: 
\begin{itemize}
\item[(1)] increasing: $X \subseteq \phi(X)$;
\item[(2)] isotone: $X \subseteq Y$ implies $\phi(X)\subseteq \phi(Y)$;
\item[(3)] idempotent: $\phi(\phi(X))=\phi(X)$.
\end{itemize}
It will be convenient for us to refer to the pair $\langle S,\phi\rangle$ of a set $S$ and a closure operator on it as a \emph{closure system}.

A subset $X \subseteq S$ is called \emph{closed} if $\phi(X)=X$. The collection of closed subsets of closure operator $\phi$ on $S$ forms a lattice, which is usually called the \emph{closure lattice} of the closure system and denoted $\op{Cl}(S,\phi)$. 

The lattice operations are denoted $\wedge$, for the \emph{meet}, and $\vee$, for the \emph{join}. 
Simultaneously, every lattice is a partially ordered set in which every two elements have a least upper bound (which coincides with the join of those elements), and a greatest lower bound (the meet).
We will use the notation $0$ for the least element of a lattice, and $1$ for its greatest element.
If $a \leq b$ in lattice $L$, then we denote by $[a,b]$ the interval in $L$, consisting of all $c$ satisfying $a \leq c\leq b$.

For every finite lattice $L$, let $\Ji(L)$ denote the set of \emph{join irreducible} elements of $L$. An element $j \in L$ is called join irreducible, if $j \not = 0$, and $j=a \vee b$ implies $a=j$ or $b=j$.

With every finite lattice $L$, we can associate a particular closure system $\langle S,\phi\rangle$ in such a way that $L$ is isomorphic to $\op{Cl}(S,\phi)$. Indeed, define a closure system with $S=\Ji (L)$ and the following closure operator:
\[ \phi(X)=[0,\bigvee X]\cap J(L), X\subseteq S.
\]
It is straightforward to check that the closure lattice of $\phi$ is isomorphic to $L$.

There are infinitely many closure systems whose closure lattices are isomorphic to a given lattice $L$. On the other hand, the closure system just described is the unique one, up to one-to-one mappings of the base sets, that satisfies two additional properties:
\begin{itemize}
\item[(1)] $\phi(\emptyset)=\emptyset$;
\item[(2)] $\phi(\{i\})\setminus \{i\}$ is closed, for every $i \in S$.
\end{itemize}
Condition (2) just says that each $\phi(\{i\})$ is join irreducible.
Note that (1) is a special case of (2), and that (2) implies the property
\begin{itemize}
\item[(3)] $\phi(\{i\})=\phi(\{j\})$ implies $i=j$, for any $i,j \in S$.
\end{itemize}
We will call a closure system with properties (1), (2) above a \emph{standard closure system}.
It is straightforward to verify that the standard system is characterized by the property that the set $S$ is of the smallest possible size. In other words, one cannot reduce $S$ to define an equivalent closure system.
There exists a standard procedure to obtain a standard closure system equivalent to a given one, see \cite{ANR11}.
We will assume that the closure systems in this paper are standard.

If $y \in \phi(X)$, then this relation between an element $y \in S$ and a subset $X \subseteq S$ in a closure system can be written in the form of an implication: $X \rightarrow y$. Thus, the closure system $\langle S,\phi\rangle$ can be given by the set of implications:
\[ 
\Sigma_\phi = \{X \rightarrow y: y \in S, X \subseteq S \text{ and } y \in \phi(X)\}.
\]
Conversely, any set of implications $\Sigma$ defines a closure system: the closed sets are exactly those subsets
$Y\subseteq S$ that \emph{respect} the implications from $\Sigma$, i.e., if $X \rightarrow x$ is in $\Sigma$, and $X \subseteq Y$, then $x \in Y$. There are numerous ways to represent the same closure system by sets of implications; those sets of implications with some minimality property are called \emph{bases}.  Thus we can speak of various sorts of bases.

It is convenient to define an implication $X\rto y$ as any ordered pair $(X,y)$, where $X \subseteq S$ and $y \in S$, especially having in mind its interpretation as a propositional formula, as in the next paragraph below. On the other hand, from the point of view of closure systems, any single implication $ X\rto x$, with $x \in X$, defines a trivial closure system, where all subsets of $S$ are closed. If such an implication is present in the set of implications $\Sigma$, then it can be removed without any change to the family of closed sets that $\Sigma$ defines. We will assume throughout the paper that implications $X\rto x$, where $x \in X$, are not included in the set of implications defining closure systems.

In general, implications $X\rto y$, where $X \subseteq S$ and $y \in S$, can be treated as the formulas of propositional logic over the set of variables $S$, equivalent to $y \vee \bigvee_{x \in X} \neg x $.  Formulas of this form are also called \emph{definite Horn clauses}. More generally, Horn clauses are disjunctions of several negative literals and at most one positive literal. The presence of a positive literal makes a Horn clause \emph{definite}.  A \emph{Horn formula} is a conjunction of Horn clauses.

There is also a direct correspondence between Horn formulas and Horn Boolean functions: a Boolean function $f:\{0,1\}^n \rto \{0,1\}$ is called a (\emph{pure} or \emph{definite}) \emph{Horn function}, if it has some CNF representation given by a (definite) Horn formula $\Sigma$.

Note that, in general, one can consider implications of the form $X \rto Y$, where $Y$ is not necessarily a one-element subset of $S$. The set $X$ is called the \emph{premise}, and $Y$ the \emph{conclusion} of an implication $X\rto Y$.  We will assume that any implication $X\rto Y$ is an ordered pair of non-empty subsets $X,Y \subseteq S$, and $Y\cap X = \emptyset$.

The following general observation about the computation of the closure via the basis is part of the folklore, see, for example, \cite{W94}.

\begin{prop}\label{inference}
Let  $\langle S,\phi\rangle$ be a closure system, with a basis $\Sigma$.  An implication $A\rto b$ holds in the closure system iff one can find  a sequence $\sigma_k = A_k\rto B_k$, $k\leq m$, of implications from $\Sigma$, such that: \emph{(I)}  $A_1 \subseteq A$; \emph{(II)} $b \in B_n$; \emph{(III)} $A_k \subseteq A \cup B_1 \cup \dots \cup B_{k-1}$ for $k>1$. 
\end{prop}

A sequence $\sigma_k$, $k \leq m$, with properties (I)-(III) from Proposition~\ref{inference} is called a $\Sigma$-\emph{inference} of $b$ from $A$. Thus, an implication $A\rto b$ holds in the closure system iff $b$ has $\Sigma$-inference from $A$, for some basis $\Sigma$.  We could also say that the implication $A\rto b$ \emph{follows} from $\Sigma$.

Following K.~Bertet and B.~Monjardet \cite{BM}, we will call the basis $\Sigma$ a \emph{unit implicational basis}, if $|Y|=1$ for all implications $X \rto Y$ in $\Sigma$. 

Given a unit basis $\Sigma$, we can replace all implications $X \rto y$ with the same premise $X$ by a single implication $X \rto Y$, where $Y$ is the union of all singletons $y$ from the conclusions of these unit implications. Such a basis will be called the \emph{aggregation} of $\Sigma$, and denoted by $\Sigma^{ag}$.

Vice versa, for every basis $\Sigma$, we may consider its \emph{unit expansion} $\Sigma_u$, where $X_i\rto Y_i$ is replaced by $\{X_i\rto y:y\in Y_i\}$.  In particular, $(\Sigma_u)^{ag}=\Sigma$, for every aggregated basis $\Sigma$.

As in \cite{ANR11}, we will call the subset $\Sigma^b=\{(A\rto B)\in \Sigma : |A|=1\}$ 
of given basis $\Sigma$ the \emph{binary part} of the basis. The \emph{non-binary} part of $\Sigma$ is $\Sigma^{nb}=\Sigma\setminus \Sigma^b$, consisting of all implications $A \rto B$ in $\Sigma$ with $|A|>1$.

Assuming that the closure system $\< S, \phi\>$ defined by $\Sigma$ is standard, we can claim that the binary relation $\geq_\phi$ on $X$ defined as:
\[ a\geq_\phi b \text{  iff  } b \in \phi(a)
\]
is a partial order. This is exactly the partial order on the join irreducible elements in $L=\op{Cl}(S,\phi)$. 

We write $|\Sigma|$ for the number of implications in $\Sigma$. 
An aggregated basis $\Sigma$ is called \emph{minimum}, if $|\Sigma|\leq |\Sigma^*|$, for any other aggregated basis $\Sigma^*$ of the same system.

The number $s(\Sigma)=|X_1|+\dots |X_n|+|Y_1|+\dots +|Y_n|$ is called the \emph{size} of the basis $\Sigma$. 
A basis $\Sigma$ is called \emph{optimum} if $s(\Sigma)\leq s(\Sigma^*)$, for any other basis $\Sigma^*$ of the system. Similarly, one can define $s_L(\Sigma)=|X_1|+\dots |X_n|$, the $L$-size, and $s_R(\Sigma)= |Y_1|+\dots +|Y_n|$, the $R$-size, of a basis $\Sigma$. The basis will be called \emph{left-optimum} (resp.~\emph{right-optimum}), if 
$s_L(\Sigma)\leq s_L(\Sigma^*)$ (resp.~$s_R(\Sigma)\leq s_R(\Sigma^*)$), for any other basis $\Sigma^*$.
Finally, the basis is called \emph{non-redundant}, if removing any implication gives a set of implications that no longer defines the same closure system. 

Now we recall the major theorem of  V. Duquenne and J.L. Guigues \cite{DG} about the canonical basis; see also N.~Caspard and B.~Monjardet \cite{CM03}. 

A set $Q\subseteq S$ is called \emph{quasi-closed} for $\<S,\phi\>$, if 
\begin{itemize}
\item[(1)] $Q$ is not closed;
\item[(2)] $Q\cap X$ is closed, for every closed set $X$ with $Q \not\subseteq X$.
\end{itemize}
In other words, adding $Q$ to the family of $\phi$-closed sets, makes the new family stable under the set intersection; in particular, it is the family of closed sets of some closure operator.

A quasi-closed set $C$ is called \emph{critical}, if it is minimal, with respect to the containment order, among all quasi-closed sets with the same closure. 
Equivalently, if $Q\subseteq C$ is another quasi-closed set and $\phi(Q)=\phi(C)$, then $Q=C$.

Let $\mathcal{Q}$ be the set of all quasi-closed sets and $\mathcal{C}\subseteq \mathcal{Q}$ be the set of critical  sets of the closure system $\<S,\phi\>$. 
Subsets of the form $\phi(C)$, where $C \in \mathcal{C}$, are called \emph{essential}.
It can be shown that by adding all quasi-closed sets to closed sets of $\<S,\phi\>$, one obtains a family of subsets stable under set intersection, thus defining a new closure operator $\sigma$.  This closure operator $\sigma$ associated with $\phi$ is called the \emph{saturation} operator. In other words, for every $Y\subseteq S$, $\sigma(Y)$ is the smallest set containing $Y$ which is either quasi-closed or closed.

\begin{thm}\label{DG} \cite{DG}
Consider the set of implications $\Sigma_C=\{C\rto ( \phi(C)\setminus C): C \in \mathcal{C}\}$. Then
\begin{itemize}
\item[(1)]  $\Sigma_C$ is a minimum basis.
\item[(2)] For every other basis $\Sigma$, for every $C\in \mathcal{C}$, there exists $(U\rto V)$ in $\Sigma$ such that $\sigma(U)=C$.
\end{itemize}
\end{thm}

The basis $\Sigma_C$ defined in Theorem~\ref{DG} is called the \emph{canonical basis} for the system $\langle S,\phi \rangle$.

We will make use of the following lemma regarding the saturation operator, due to M.~Wild. 

\begin{lem} \label{2-3-extended} \cite{W94}
Let $\Sigma$ be a basis for the closure system $\langle S,\phi \rangle$, and let $U \rto V$ be an implication in $\Sigma$.  Let $\Sigma' = \{ (X \rto Y) \in \Sigma : \phi(X) = \phi(U) \}$.
For any subset $W \subseteq S$ such that $\phi(W) \subseteq \phi(U)$, the implication $W\rto \sigma(W)$ follows from $\Sigma \setminus \Sigma'$.
\end{lem}

A big part of the current paper is devoted to the closure systems that we call \emph{$UC$-systems}.  In such a system every essential element $X$ has exactly one critical set $C\subseteq S$ with $\phi(C)=X$. The source of inspiration for $UC$-systems is its proper subclass of closure systems whose closure lattices satisfy the join-semidistributive law.

A lattice is called \emph{join-semidistributive} if it satisfies the lattice law 
\[
(SD_\vee) \qquad
x\vee y=x\vee z \rto x\vee y= x\vee (y\wedge z).
\]
The join-semidistributive law plays an important role in lattice theory, for example in the study of free lattices, see \cite{FJN}.

An important subclass of finite \jsd\ lattices are so-called lattices \emph{without} $D$-\emph{cycles}. First, we need to define the $D$-relation on $\Ji (L)$. If $x\leq \bigvee_{i\in I} x_i$ for $x,x_i \in \Ji(L)$, and $x \not \leq x_i$ for all $i \in I$, then $X=\{x_i:i\in I\}$ is called a \emph{non-trivial cover} of $x$. For any $X,Y\subseteq L$, we say that $Y$ refines $X$, and write $Y\ll X$, if for every $y \in Y$ there exists  $x \in X$ such that $y\leq x$. The set $X\subseteq \Ji (L)$ is called a \emph{minimal cover} for $x$, if $Y\ll X$ can be a cover of $x$, only if $X\subseteq Y$. In other words, no $x_i \in X$ can be deleted, or replaced by a set of join irreducibles $Z$ with $z< x_i$ for all $z \in Z$, to obtain another cover for $x$. Finally, a binary relation $D$ is defined on $\Ji (L)$: $xDy$ iff $y\in X$ for some minimal cover $X$ of $x$.

We say that a lattice $L$ is \emph{without $D$-{cycles}} if it does not have $x_1,\dots, x_n \in \Ji (L)$, where $n >2$, such that $x_iDx_{i+1}$ and $x_1=x_n$. It is well-known that every finite lattice without $D$-cycles is \jsd\!. In lattice literature, the lattices without $D$-cycles are known as \emph{lower bounded}. In section \ref{NP} we will also briefly mention \emph{bounded} lattices: a lattice $L$ is bounded if both $L$ and the dual lattice $L^\delta$ are without $D$-cycles.

\section{Minimum unit basis versus optimum basis}\label{min-opt}

The following result of D. Maier \cite{Mai} establishes the connections between different types of effective bases.

\begin{thm}
Every optimum basis is minimum, and every minimum basis is non-redundant.
\end{thm}
While the second implication is rather straightforward, the first one requires a non-trivial argument. Moreover, neither of these two statements can be reversed.

Given any basis $\Sigma$, we call it a \emph{minimum unit} basis, if its unit expansion $\Sigma_u$ has the minimum number of (unit) implications among all possible unit bases for this closure system. Being a minimum unit basis is equivalent to having right-side optimization: $s_R=|Y_1|+\dots +|Y_n|$ is minimum among all bases for this system. 
The problem of determining a minimum unit basis is associated with the problem of finding the minimum CNF-representation of Boolean functions, or minimum representation of a directed hypergraph, see E.~Boros et al. \cite{B11}.

It turns out that minimum unit bases and optimum bases of any closure system have an intimate connection. 
Our goal in this section is to extend Maier's result to include minimum unit bases into a hierarchy: \emph{every optimum basis is right-side optimum}.


An important part of the statement is that, in every optimum basis, the left side of every implication has a \emph{fixed size} $k_C$, $C \in \mathcal{C}$, that does not depend on the choice of the optimum basis.  This makes it into a parameter of the closure system itself. 

\begin{thm} \label{W}
Let $\<S,\phi\>$ be a closure system.
\begin{itemize}
\item[(I)] If $\Sigma'$ is a non-redundant basis, then $\{\sigma(U): (U\rto V) \in \Sigma'\}\subseteq \mathcal{Q}$.
\item[(II)] Let $\Sigma_O$ be an optimum basis. For any critical set $C$, let $X_C\rto Y_C$ be an implication from this basis with $\sigma(X)=C$. Then $|X_C|=k_C:=min\{|U|: U\subseteq C, \phi(U)=\phi(C)\} \,=\, min\{|U|: U\subseteq C, \sigma(U)=C\}$.
\end{itemize}
\end{thm}

For the rest of this section we assume that $\Sigma_u$ is a minimum unit basis of the closure system $\<S,\phi\>$, and $\Sigma=(\Sigma_u)^{ag}$ is the aggregation of $\Sigma_u$.
The following two lemmas borrow from the argument of Theorem 5 in M.~Wild \cite{W94}.
   
\begin{lem}\label{min}
There exists a minimum basis $\Sigma^*$ such that $|\Sigma^*_u|=|\Sigma_u|$.
\end{lem}
\begin{proof}
We start from $\Sigma$. Since $\Sigma_u$ is a minimum basis, $\Sigma$ is non-redundant. 

According to Theorem \ref{DG} (2), for each critical set $C$, we will be able to find an implication $(X_C\rto Y_C)\in \Sigma$, such that $\sigma(X_C)=C$. Suppose $\Sigma$ has another implication $U\rto V$ distinct from all $X_C\rto Y_C$, $C\in \mathcal{C}$. By Theorem \ref{W} (I) we have $\sigma(U)\in \mathcal{Q}$, and thus we can find a critical set $C_0\in \mathcal{C}$ such that $C_0\subseteq \sigma(U)$ and $\phi(C_0)=\phi(U)$.  The premise of the implication $X_{C_0} \rto Y_{C_0}$ from $\Sigma$ also satisfies
 $\phi(X_{C_0})=\phi(C_0)=\phi(U)$. 

By Lemma \ref{2-3-extended}, the implication $U\rto \sigma(U)$ follows from the set of implications $\Sigma\setminus\{X_{C_0}\rto Y_{C_0}, U\rto V\}$.  
Form $\Sigma^\circ=(\Sigma\setminus\{X_{C_0}\rto Y_{C_0}, U\rto V\})\cup \{X_{C_0}\rto Y_{C_0}\cup V\}$. Then $U\rto V$ follows from $\Sigma^\circ$. Moreover, $|\Sigma^\circ_u| \leq |\Sigma_u|$, but due to the minimality of $\Sigma_u$, we have $|\Sigma^\circ_u|= |\Sigma_u|$. 

Repeat this procedure, replacing each implication $U\rto V$ in $\Sigma^\circ$ distinct from all $X_C\rto Y_C$, obtaining a basis $\Sigma^*$. It will have the same number of implications as the canonical basis $\Sigma_C$, whence it is minimum. Moreover, $|\Sigma^*_u|=|\Sigma_u|$.
\end{proof}


Lemma \ref{min} suggests the following procedure for dealing with the problem of unit basis minimization.

\begin{cor}
Given a unit basis $\Sigma_u$ with the non-redundant aggregation $\Sigma$, in time polynomial in $s(\Sigma)$, one can build a unit basis $\Sigma_u^*$ such that $|\Sigma_u^*|\leq |\Sigma_u|$, while the aggregation $\Sigma^*$ of\/ $\Sigma_u^*$ is a minimum basis.
\end{cor}
Indeed, it follows from the proof of Lemma \ref{min} that there exists a partition of $\Sigma=\bigcup_{C \in \mathcal{C}} \Sigma_C$ into blocks $\Sigma_C$ with respect to critical sets $C \in \mathcal{C}$: if $(A\rto B)\in \Sigma_C$, then $C\subseteq \sigma(A), B\subseteq \phi(C)$. Moreover, there exists at least one implication $(D\rto F) \in \Sigma_C$ such that $C=\sigma (D)$. Thus, the whole block can be replaced by $D\rto F'$, where $F'$ is the union of all right side sets of implications in $\Sigma_C$. Note that it takes polynomial time in $s(\Sigma)$ to build the partition $\Sigma=\bigcup_{C \in \mathcal{C}} \Sigma_C$. In particular, every minimum unit basis $\Sigma_u$ can be turned into a minimum unit $\Sigma_u^*$, whose aggregation is minimum. 

\begin{lem}\label{premise}
There exists a minimum basis $\Sigma^{**}$ such that, for every $(X_C\rto Y_C) \in \Sigma^{**}$, we have $|X_C|=k_C$. Moreover, $|\Sigma^{**}_u|=|\Sigma_u|$.
\end{lem}
\begin{proof}
Start with $\Sigma^*$ obtained in Lemma \ref{min}. Then $\sigma(X_C)=C$, and there exists $X_C^*$ with $\sigma(X_C^*)=C$ and $|X_C^*|=k_C$, as defined in Theorem \ref{W}(2). An implication $ X_C\rto C$ follows from $\Sigma^*\setminus \{X_C\rto Y_C\}$, due to Lemma \ref{2-3-extended}. Hence, $X_C\rto Y_C$ follows from $(\Sigma^*\setminus\{X_C\rto Y_C\})\cup \{X_C^*\rto Y_C\}$. Repeat this procedure for each implication in $\Sigma^*$, and obtain a new basis $\Sigma^{**}$. Since only the premise of each implication may change,
$|\Sigma^{**}_u|=|\Sigma^*_u|=|\Sigma_u|$.
\end{proof} 

The basis obtained as the result of previous two lemmas has a nice property.
\begin{lem} \label{opt}
$\Sigma^{**}$ is an optimum basis.
\end{lem}
\begin{proof}
Suppose $\Sigma^{**}=\{X_C\rto Y_C: C \in \mathcal{C}\}$ is not optimum. Consider an optimum basis $\Pi=\{A_C\rto B_C: C \in \mathcal{C}\}$, so that $s(\Pi)<s(\Sigma^{**})$. According to Theorem \ref{W}(II), $|A_C|=k_C=|X_C|$. It follows that $ \Sigma(|B_C|: C\in \mathcal{C})< \Sigma(|Y_C|: C \in \mathcal{C})$. Both sums represent the number of implications in the unit expansion of each basis. Thus, $|\Pi_u|<|\Sigma^{**}_u|=|\Sigma_u|$, which contradicts to the fact that $\Sigma_u$ is a minimum unit basis.
\end{proof}

Now we can prove the main result of this section.

\begin{thm}\label{main}
There exists a minimum unit basis whose aggregation is an optimum basis. The unit expansion of every optimum basis is a minimum unit basis.
\end{thm}
\begin{proof}
According to Lemma \ref{opt}, the unit expansion of the basis $\Sigma^{**}=\{X_C\rto Y_C: C \in \mathcal{C}\}$ has the same number of implications as $\Sigma_u$, hence, $\Sigma^{**}_u$ is also a minimum unit basis. Besides, $(\Sigma^{**}_u)^{ag}=\Sigma^{**}$, which is optimum. This proves the first statement.
Take another optimum basis $\Pi=\{A_C\rto B_C: C \in \mathcal{C}\}$. Then $s(\Pi)=s(\Sigma^{**})$. Moreover,
 $|A_C|=k_C=|X_C|$, for each $C\in \mathcal{C}$, hence, $ \Sigma(|B_C|: C\in \mathcal{C})=\Sigma(|Y_C|: C \in \mathcal{C})$. The latter equality means that the unit expansions of both bases have the same number of implications. Hence, $\Pi_u$ is also a minimum unit basis, and the second statement is also proved.
\end{proof}

\begin{cor}\label{lr}
A basis $\Sigma$ is optimum if and only if it is left-side optimum and right-side optimum.
\end{cor}
\begin{proof}
It follows from Theorem \ref{W} (II) that every optimum basis is left-side optimum. It was shown in Theorem \ref{main} that every optimum basis is a minimum unit basis, equivalently, it is right-side optimum.

Vice versa, let $\Sigma$ be both left-optimum and right-optimum and let $\Sigma_O$ be any optimum basis. Recall that $s(\Sigma_O)=s_L(\Sigma_O) + s_R(\Sigma_O)$. It follows from left- and right-optimality of $\Sigma$ that $s_L(\Sigma) \leq s_L(\Sigma_O)$ and $s_R(\Sigma) \leq s_R(\Sigma_O)$. Hence, $s(\Sigma) \leq s(\Sigma_O)$, which implies $s(\Sigma) = s(\Sigma_O)$ and $\Sigma$ is optimum.
\end{proof}

It was proved in D.~Maier \cite{Mai} that finding an optimum basis of a given closure system is an NP-complete problem. A similar result about the minimum unit basis in the form of minimal directed hyper-graph representation was established later in G.~Ausiello  et al.\cite{AAS86}. It follows from Theorem \ref{main} that the Maier's result can be obtained for free from the latter.

In \cite{W00}, M.Wild proved that, given any basis $\Sigma$ of a closure system whose lattice of closed sets is \emph{modular}, one can obtain an optimum basis of this system in time $O(s(\Sigma)^2)$. Moreover, it was shown in C.~Herrmann and M.~Wild \cite{HW96} that the test for the modularity of the closure system can also be achieved in polynomial time. This has the following consequence.

\begin{cor}
Let $\<S,\phi\>$ be a finite closure system, whose lattice of closed sets is modular. Then there is a polynomial time algorithm to obtain a minimum unit basis for this system. 
\end{cor} 
Indeed, according to Theorem \ref{main}, this basis is the unit expansion of the optimum basis built in Wild \cite{W00}.

\section{Optimizing binary part of a basis}\label{bin}

In this section we touch upon the issue of optimization of the binary part of any basis. It turns out that the binary part is independent of the non-binary part, in the sense that there are parameters of optimization for the binary part pertinent to a closure system itself, and independent of the basis that is considered. These parameters are similar to the parameters $k_C$ of left-side optimization considered in Theorem \ref{W} (II).

For this section in particular, it is important that we consider a \emph{standard} closure system $\<S,\phi\>$, where it is assumed that $\phi(\{a\})\setminus \{a\}$ is a closed set, for every $a \in S$. This set will be denoted $A_*$. 

First, we observe that singletons from $S$ that are not closed are always represented in the binary part of any basis.

\begin{lem}
Let $\Sigma$ be any aggregated basis of a standard closure system  $\<S,\phi\>$.  Then $(a\rto B)\in \Sigma$ for some nonempty $B \subseteq S$ iff $\phi(\{a\})\setminus \{a\} \not = \emptyset$.
\end{lem}
\begin{proof} If $\phi(\{a\})\setminus \{a\} \not = \emptyset$, then $\{a\}$ is a critical set for the essential set $\phi(\{a\})$.  Thus $(a\rto B)$ must be in  $\Sigma$, due to Theorem \ref{DG} (2). The converse statement is obvious, since $\Sigma$ defines the closure system $\<S,\phi\>$ and $B \subseteq \phi(\{a\})$. Besides, $B\not = \{a\}, \emptyset$, according to the definition of implication.
\end{proof}

Secondly, we would like to introduce the bases with properly distinguished binary and non-binary parts. Some bases may have implications in $\Sigma^{nb}$ that should truly belong to $\Sigma^b$. 

\begin{df}\label{regD}
Let $\< X, \phi\>$ be a closure system. An aggregated basis $\Sigma$ is called \emph{regular}, if
$\Sigma^{nb}$ does not have implications $\{a\}\cup F\rto D$ with $F\subseteq \phi(\{a\})$.
\end{df}

In other words, in a regular basis $\Sigma$, no implication $A\rto B\in \Sigma^{nb}$ is equivalent to $a\rto B$, for any $a \in A$.

It is easy to observe that both the aggregated the canonical basis and the $D$-basis of \cite{ANR11}  are regular by the definition. The concept of the regular basis is highly important in the key Definition \ref{delta} of section \ref{D-rel}.


\begin{prop}\label{reg prop}
Let $\< X, \phi\>$ be a standard closure system. If\/ $\Sigma$ is regular, then for every implication $a\rto B$ in $\Sigma^b$ it holds that $\phi(B)= \phi(\{a\})\setminus \{a\}$.
\end{prop}
\begin{proof} Since $\< X, \phi\>$ is standard, $\phi(\{a\})\setminus \{a\}$ is closed. 
To simplify the notation, we will denote $A_*=\phi(\{a\})\setminus \{a\}$.

If $a \rto B$ is in $\Sigma$, then $B\subseteq \phi(\{a\})$ and $a \not \in B$, whence $B\subseteq A_*$.
Since $A_*$ is closed, this implies $\phi(B)\subseteq A_*$. We want to verify that $\phi(B)=A_*$.

Consider any $c \in A_*$.
Since $c \in \phi(\{a\})$, by Proposition \ref{inference}, there should be a sequence of implications $\sigma_1,\dots, \sigma_k \in \Sigma$, say $\sigma_i=A_i\rto B_i$, such that $A_1=\{a\}$, $c \in B_k$, and $A_i\subseteq A_1\cup B_1\cup\dots \cup B_{i-1}$ for $1<i\leq k$.   Since $\Sigma$ is aggregated, we have $B_1 \subseteq \Phi(B)$.  As $\Sigma$ is regular, $a\notin A_2,\dots, A_k$, whence by induction it follows that $B_i\subseteq \phi(B)$.  For $i=k$ this yields $c \in B_k \subseteq \phi(B)$, so that $c \in \phi(B)$, as desired.
\end{proof}

The next lemma shows that any basis can be modified to become regular, without increasing its size.

\begin{lem}\label{reg}
If $\Sigma$ is any aggregated basis of a standard closure system, then one can find a new basis $\Sigma_r$ that is regular, $|\Sigma_r|\leq |\Sigma|$ and $s_L(\Sigma_r)\leq s_L(\Sigma)$, $s_R(\Sigma_r)\leq s_R(\Sigma)$. 
\end{lem}
\begin{proof}
Let $\{a\}\cup F\rto D$ in $\Sigma$, where $F\subseteq \phi(\{a\})$. By Proposition \ref{inference}, $\Sigma^b$ should have an implication $a \rto B$, for some $B$. Consider the splitting of $D$: $D=(D\setminus \phi(F))\cup  (D\cap \phi(F))$. If both subsets in the union are non-empty, then one can replace this implication by
$a\rto D\setminus \phi(F)$ and $F \rto (D\cap \phi(F))$ without changing $s_L$ and $s_R$. If either of two subsets is empty, then the original implication can be replaced by one of the two new implications, correspondingly. Then $a\rto B$ can be aggregated with $a\rto D\setminus \phi(F)$, reducing $s_L$ and, possibly, $s_R$ as well. Observe that this procedure will not increase the number of implications.
\end{proof}

\begin{cor} Every optimum basis is regular.
\end{cor}

We note that the procedure in Lemma \ref{reg} does modify the non-binary part of a given basis $\Sigma$.

\begin{exm}
Consider a closure system be given by the basis $\Sigma=\{a\rto b, ab\rto c, bc\rto d\}$. The second implication has the premise $ab\subseteq \phi(\{a\})$, and it follows from $a\rto c$. Hence, the regular basis obtained via procedure of Lemma \ref{reg} is $\Sigma_r=\{a\rto bc, bc\rto d\}$. We note that $|\Sigma_2|=2<3=|\Sigma|$, $s_L(\Sigma_r)=3<5=s_L(\Sigma)$, and $s_R(\Sigma_r)=s_R(\Sigma)=3$.
\end{exm}

The binary part $\Sigma^b$ of some standard basis will also be called \emph{regular}, if it satisfies the condition of Proposition \ref{reg prop}: for every implication $a\rto B$ in $\Sigma^b$ it holds that $\phi(B)= \phi(\{a\})\setminus \{a\}$. 

The next statement shows that the binary part of any basis can be replaced by any regular binary part without any modification to the non-binary part.

\begin{lem}\label{bin opt}
Let $\Sigma=\Sigma^b\cup \Sigma^{nb}$ be any basis of a standard closure system. One can replace any $(a\rto D)\in \Sigma^b$ by arbitrary $a\rto B$, with $\phi(B)=\phi(\{a\})\setminus \{a\}$, to obtain a new basis $\Sigma_*$ that preserves the non-binary part: $\Sigma_*^{nb}=\Sigma^{nb}$.
\end{lem}
\begin{proof}
The argument proceeds by induction on the height of the element $a$ in the partially ordered set $(S,\geq_\phi)$. By the \emph{height} of an element $x$ in a partial ordered set $(S,\geq)$, we mean the length $k$ of the longest chain $x = x_0>x_1 \dots > x_k$ in $(S,\geq)$.

The argument for $k=1$ and the induction step are similar, so we combine both cases. Assume that $a\rto B$ is in $\Sigma_*^b$ and that $a \rto D$ is in $\Sigma^b$, where either $a$ is of height $1$, or of height $n+1$, and it is already proved that any implication $(b\rto F)\in \Sigma^b$ with height at most $n$ follows from $\Sigma_*$. If $a$ is of the height $1$, then every $c \in \phi(\{a\})$ is an \emph{atom}, i.e., $\{c\}=\phi(\{c\})$. In particular, $c$ cannot be a premise of any binary implication.

We want to show that $a\rto D$ follows from $\Sigma_*$. Let use again the notation $A_*=\phi(\{a\})\setminus \{a\}$.

Assume there exists $d \in D\setminus B$. Since $d \in A_*=\phi(B)$, the implication $B\rto d$ should follow from $\Sigma$. Then one can find a sequence of implications $\sigma_1,\dots, \sigma_k \in \Sigma$, $\sigma_i=A_i\rto B_i$, such that $B\subseteq A_1$, $d \in B_k$, and $A_i\subseteq A_1\cup B_1\cup \dots \cup B_{i-1}\subseteq A_*$. If $a$ is of height $1$, then none of these implications is binary, since all elements in $A_*$ are atoms. If $a$ is of height $n+1$, then all binary implications are of the form $(b \rto F)\in \Sigma^b$, where $b \in A_*$, whence $b$ is of height $\leq n$, and such implications follow from $\Sigma_*$ by the inductive hypothesis. Place $\sigma_0=a\rto B$ prior to the first implication in inference  $\sigma_1,\dots, \sigma_k$, and replace each binary implication among $\sigma_i$ by its inference from $\Sigma_*$. This will give the inference of  $a\rto d$ from $\Sigma_*$.   
\end{proof}

We note that similar statement is proved in Corollary 5.7 of E.~Boros et al. \cite{B10}, where an exclusive set of implications may, in particular, be any subset of the basis, whose left and right sides are contained in some $\phi(Y)$. In our case, it is $\phi(\{a\})$. 

The consequence of two lemmas is the following statement about the right-side optimization of the binary part.

\begin{thm}\label{rs-bin}
Let $\Sigma_C$ be the canonical basis of a standard closure system $\<S, \phi\>$, and let $x_C\rto Y_C$ be any binary implication from $\Sigma_C$. Every (regular right-side) optimum basis $\Sigma$ will contain an implication $x_C\rto B$, where $|B|=b_C=\min \{|Y|: \phi(Y)=\phi(\{x_C\})\setminus\{x_C\}\}$.
\end{thm}
\begin{proof}
Suppose we are given a (regular right-side) optimum basis $\Sigma$. 
For each $(x_C\rto Y_C) \in \Sigma_C^b$, there should be $(x_C \rto D) \in \Sigma$; besides, $\phi(D)=\phi(\{x_C\})\setminus\{x_C\}$. Suppose there exists $B$ such that $\phi(B)=\phi(\{x_C\})\setminus\{x_C\}$ and $|B|<|D|$. Then, by Lemma \ref{bin opt}, it would be possible to replace $x_C\rto D$ by $x_C\rto B$, without changing other implications of the basis. This would reduce  $s_R$, a contradiction with the right-side optimality of $\Sigma$. Hence, $|D|$ is minimal among subsets whose closure is  $\phi(\{x_C\})\setminus\{x_C\}$.
\end{proof}

We will call a basis $B$-\emph{optimum}, if the size of its binary part is minimum among all possible \emph{regular} bases of the closure system $\<S, \phi\>$.

\begin{cor}\label{Bopt}
Every optimum basis is $B$-optimum. Moreover, every $B$-optimum basis $\Sigma$ has the size of its binary part $s(\Sigma^b)=|\Sigma_C^b|+\Sigma_{|C|=1}b_C$.
\end{cor}
\begin{proof}
According to Theorem \ref{rs-bin}, every optimum basis has the smallest right size of its binary part, among the same measurement of all other regular bases. The premises in the binary parts are just singletons, so the total size of the premises in every aggregated basis $\Sigma^b$ is $|\Sigma^b|$.  
\end{proof}

We can point some easy computable lower bound  for parameters $b_C$. Recall that by an \emph{extreme point} of a closed set $X$ one calls an element $x \in X$ such that $x \not \in \phi(X\setminus \{x\})$. The set of extreme points of $X$ is denoted by $Ex(X)$.

\begin{cor}\label{Ex} Let $x_C\rto B$ be an implication from any optimum basis of a closure system  $\<S, \phi\>$, and let $X_*$ denote $\phi(\{x_C\})\setminus \{x_C\}$. Then $Ex(X_*) \subseteq B$. In particular, $|B|=b_C \geq |Ex(X_*)|$.
\end{cor}
\begin{proof} Since an optimum basis is regular, $\phi(B)=X_*$. Suppose $y \in Ex(X_*)\setminus B$. Then $B\subseteq X_*\setminus \{y\}$, whence $\phi(B)\subseteq \phi(X_*\setminus \{y\})$ and $y \not \in \phi(B)$, a contradiction. Therefore, $Ex(X_*) \subseteq B$.
\end{proof}

In general, it is not true that $\phi(Ex(Y))=Y$, for any $Y\subseteq S$. Thus, it is possible that $b_C > |Ex(X_*)|$. It is easy to observe that this lower bound is attained in closure systems called \emph{convex geometries}, so that the optimum binary part is tractable in such closure systems. This is treated with more detail in K.~Adaricheva \cite{A12}.
\vspace{0.2cm}

One natural choice for the binary implications in a basis $\Sigma$ is the \emph{cover relation} $\succ_\phi$ of $\geq_\phi$. (By the definition, $a\succ_\phi b$ means that $a\geq_\phi b, a\not = b$ and $a\geq_\phi c\geq_\phi b$ implies $c=a$ or $c=b$.) Namely, $a\rto A_\succ$ is included into $\Sigma^b$, if $A_\succ =\{b \in \phi(\{a\}): a\succ_\phi b\}$.
It is straightforward to verify that $\phi(A_\succ )= \phi(\{a\})\setminus\{a\}$. Hence, such a binary part can be included into any basis, irrelevant to its non-binary part. It is algorithmically easy to compute $A_\succ$, see \cite[Proposition 16]{ANR11}. Also, the advantage of this choice is that every binary implication $a\rto b$ that follows from $\Sigma$, follows from binary part only. Of course, it is by no means guaranteed that $A_\succ$ has the minimal cardinality required in an optimum basis. It can even be reducible, i.e., some proper subset of $A_\succ$ can still have the same closure as $A_\succ$. 

\begin{exm}\label{cover}
Consider the closure system given by the canonical basis $\Sigma_C=\{a\rto bcd, bcdy\rto a,bc\rto d\}$. Apparently, the only binary implication $a\rto bcd$ is presented by $\succ_\phi$, i.e., $A_\succ=\{b,c,d\}$. On the other hand, we can reduce this implication to $a\rto bc$, since $\phi(\{b,c\})=\phi(\{a\})\setminus\{a\}$.
\end{exm}

We will consider in section \ref{Kbas SD} another form of the basis, for the  $UC$-closure systems, whose binary part will be non-redundant.

\section{K-basis in general closure systems}\label{kbas}

As follows from the results of sections \ref{min-opt} and \ref{bin}, an optimum basis is the most desirable among all bases of the closure system, since it implies other forms of minimality.

On the other hand, finding an optimum basis is generally a computationally hard problem. In this section we describe a new notion of a $K$-basis which can be computed from the canonical basis in polynomial time and provides reduction in size, while maintaining the minimum number of implications. 

Let $\Sigma_C=\{C\rto Y_C: C \in \mathcal{C}\}$ be a canonical basis. We will call any basis $\Sigma^*=\{X^*_C\rto Y^*_C: C \in \mathcal{C}\}$
a \emph{refinement} of the canonical basis if
$X^*_C\subseteq C$, $Y^*_C\subseteq Y_C$, and $\phi(X^*_C)=\phi(C)$ for all $C \in \mathcal C$. 
Every refinement is a minimum basis, while $s(\Sigma^*)\leq s(\Sigma_C)$.

According to Theorem \ref{DG} (2) and Theorem \ref{W} (II), every optimum basis is a refinement of the canonical. Besides, for every critical set $C$, an optimum basis would have an implication $X^*\rto Y^*$ with $|X^*|=k_C$, the minimum cardinality of a subset $X\subseteq C$ for which $\phi(X)=\phi(C)$ (equivalently, $\sigma(X)=C$). Such set $X$ is called a \emph{minimal generator} for the critical set $C$.

The $K$-basis is built on the idea of minimizing the left side of implications of the canonical basis with respect to all special \emph{order ideals} contained in critical sets. Thus, the left sides of the $K$-basis are not necessarily of the smallest cardinality. On the other hand, the algorithm producing a $K$-basis is fast and easy.

We observe that, according to the definition of a critical set, every $C$ is an $\geq_\phi$-order ideal, i.e., if $a \in C$ and $a\geq_\phi b$, then $b\in C$. Among all order ideals $Y\subseteq C$ with the property $\phi(Y)=\phi(C)$, we can find a minimal one with respect to containment. Finally, for every order ideal $Y$, it is clear that $\phi(\op{max}(Y))=\phi(Y)$, where $\op{max}(Y)$ is the subset of $\geq_\phi$-maximal elements of $Y$.

We will call a set $X^*$ a \emph{minimal order generator} of the critical set $C$ if $X^*=\op{max}(X)$ for some minimal (with respect to containment) order ideal $X \subseteq C$ such that $\phi(X)=\phi(C)$.

\begin{prop}\label{min gen} Let $C$ be a critical set for the closure system $\< S, \phi \>$. Build a sequence $X_0=C, X_1, \dots, X_n$ of subsets of $C$ as follows: (1) $X_{k+1}=X_k\setminus \{x_k\}$, where $x_k$ is any $\geq_\phi$-maximal element of $X_k$; (2) $\phi(X_k)=\phi(C)$; (3) none of the maximal elements of $X_n$ can be removed to obtain a set with the same closure as $C$.  Then $C_K:=\op{max} (X_n)$, the set of $\geq_\phi$-maximal elements of $X_n$, is a minimal order generator, and every minimal order generator may be obtained via such a sequence.
\end{prop} 
\begin{proof} 
One can reach any minimal order ideal $X_n$ contained in $C$ and with $\phi(X_n)=\phi(C)$, through the series of steps, when a maximal element of an order ideal is removed, producing another order ideal with the same closure, which is exactly the procedure described in the proposition. 
\end{proof}

To align the notion of a minimal order generator with existing lattice terminology we include the following tautological statement.

\begin{prop}\label{12}
Every minimal order generator $C_K\subseteq C$ is a $\ll$-minimal join representation of the element $\phi(C)$ in the closure lattice of $\< S, \phi \>$.
\end{prop} 

A critical set might have several minimal order generators, and the procedure described in Proposition \ref{min gen} may lead to any of them, depending on the order in which the maximal elements are removed. See Example \ref{2Kbases} for an illustration. 

\begin{exm}
Not every minimal generator for a critical set $C$ will be simultaneously a minimal order generator.  
Both are sets satisfying $\phi(X)=\phi(C)$, but the former type is minimal with respect to $\subseteq$, while the latter is minimal with respect to $\ll$.  A concrete example is the lattice in Example \ref{B4double}.  
\end{exm}




\begin{df}\label{Kbs}
A set of implications $\Sigma_K$ is called a $K$-\emph{basis} if it is obtained from the canonical basis $\Sigma_C$ by replacing each implication $(C\rto Y_C)\in\Sigma_C$ by $C_K\rto Y_K$, where $C_K$ is a minimal order generator of $C$, and $Y_K=\op{max} (Y_C)$.
\end{df}

We summarize easy facts about $K$-bases.
\begin{lem}
Let $\Sigma_K$ be a $K$-basis for $\langle S,\phi \rangle$.
\begin{itemize} 
\item[(1)] $\Sigma_K^b$ represents the cover relation $\succ_\phi$ of $\geq_\phi$; 
\item[(2)] $\Sigma_K$ is a basis of $\< S, \phi \>$;
\item[(3)] $\Sigma_K$ is a minimum basis and $s(\Sigma_K)\leq s(\Sigma_C)$.
\end{itemize}
\end{lem}
\begin{proof}
First, note that if $x\geq_\phi y$ and $y\not = x$, then $(x\rto Y)\in \Sigma_C$ and $y \in Y$, due to definition of $\Sigma_C$. Secondly, $\{x\}$ will be a minimal order generator for itself and $Y$ comprises all elements $y \in S$ with $x\geq_\phi y$. Choosing only $\geq_\phi$-maximal elements $Y_K\subseteq Y$ will give exactly the lower covers of $x$, which proves (1).

Since every $a\geq_\phi b$ is in the transitive closure of cover sub-relation of $\geq_\phi$, every $(x\rto Y)\in \Sigma_C^b$ follows from $\Sigma_K^b$. Also, $C_K\rto Y_C$ follows from $C_K\rto Y_K$ and $\Sigma_K^b$.
Finally, $(C\rto Y_C)\in \Sigma_C$ follows from $C_K\rto Y_C$, since $C_K\subseteq C$. Hence, $\Sigma_C$ follows from $\Sigma_K$.

Vice versa, $C_K\rto Y_C$ follows from $\Sigma_C$, due to the definition of $C_K$. Moreover, $C_K\rto Y_K$ follows from $C_K\rto Y_C$, since $Y_K\subseteq Y_C$, which finishes the proof of (2).

Part (3) easily follows from the definition of $\Sigma_K$, since $C_K\subseteq C$ and $Y_K\subseteq Y_C$. 
\end{proof}

The following example demonstrates the advantage of considering $K$-bases due to possible size reduction of the canonical basis, while preserving its property of the minimality. The other purpose of this example is to show that there are possibly several $K$-bases associated with a closure system.

\begin{exm}\label{2Kbases}
Consider the standard closure system on $S=\{x,y,z,e,d,u\}$ given by its closure lattice in Figure 1.
Consider the critical set $C=S\setminus \{e\}$, for which $\phi(C)=S$. Following procedure in Proposition \ref{min gen}, we would have two maximal elements $x,y \in C$ that can be removed: $X_1=C\setminus \{x\}$, $X_2=C\setminus \{y\}$, $\phi(X_1)=\phi(X_2)=\phi(C)$.

Thus, $(C\rto e)\in \Sigma_C$ can be refined to either $yd\rto e$ ($\op{max} X_1=\{y,d\}$) or $xd\rto e$ (($\op{max} X_2=\{x,d\}$), and this system has two $K$-bases.

Here is the first of $K$-bases:\\
$y\rto u, z\rto u, d\rto z, e\rto d, yd\rto e, xu\rto y, zy\rto x$, with $s(\Sigma_K)=17$.

In the second one, the implication $yd\rto e$ is replaced by $xd\rto e$, which does not change the size of the basis.

Compare both with $\Sigma_C$:\\
$y\rto u, z\rto u, d\rto zu, e\rto dzu, xyzdu\rto e, xu\rto y, zyu\rto x$, with $s(\Sigma_C)=24$.
\end{exm}

\begin{figure}[htbp]\label{pic1}
\begin{center}
\includegraphics[height=2.5in,width=6.0in]{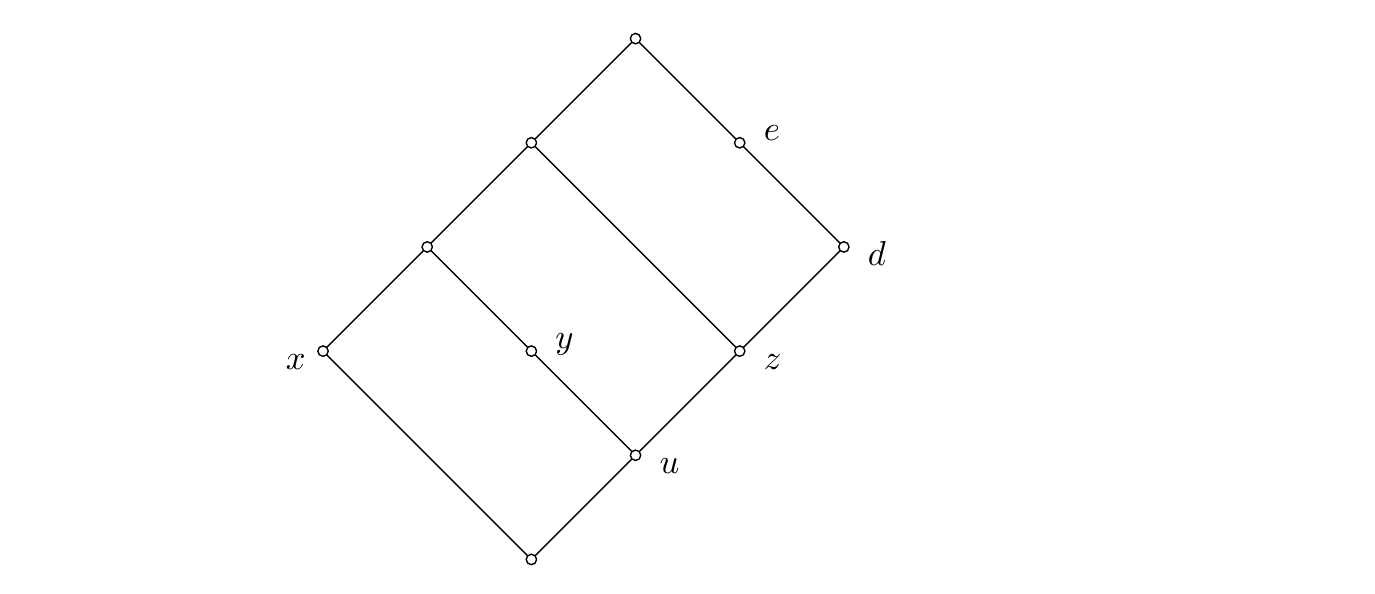}
\caption{Example \ref{2Kbases}}
\end{center}
\end{figure}

While searching for a minimal order generator of smallest cardinality, for each critical set $C$, could be a long process, finding \emph{some} minimal order generator is a fast quadratic algorithm. 

Let $\Sigma_C$ be a canonical basis of a standard closure system $\< S, \phi \>$. Assume that $|X|=n$, $|\Sigma_C|=m$, $s(\Sigma_C)=k$ and $s(\Sigma_C^b)=k_b$.
It was shown in Theorem 11.3 of \cite{FJN} that it will take time $O(nk_b+n^2)$ to produce the cover relation of the partially ordered set $(S,\geq_\phi)$. Moreover, an upper bound for the number of pairs in the cover relation is $k_b$.

\begin{prop}\label{Ktime}
Given the canonical basis $\Sigma_C$ of a standard closure system with $s(\Sigma_C)=k$, it will take time $O(k^2)$ to produce one minimal order generator for each $C\in \mathcal{C}$. Thus, it requires time $O(k^2)$ to produce a $K$-basis. 
\end{prop}
\begin{proof}
 For each critical set $C$, one would need at maximum $k_b$ steps to recognize a maximal element $x\in C$. Then, the time $O(k)$ is needed to check whether $\phi(C\setminus \{x\})=\phi(C)$ (using the forward chaining procedure, for example). If this is true, then $X_1=C\setminus \{x\}$, and one proceeds with $X_1$ in place of $C$. Otherwise, we would search during time $O(k_b)$ for another maximal element of $C$. The number of maximal elements checked overall cannot be more than $|C|$, and the time spent on each of them is $O(k)$. The overall time is the summation over all critical sets, i.e., $O(k^2)$.
\end{proof}

\section{$D$-relation from the canonical basis}\label{D-rel}

The goal of this section is to show that the $D$-relation defined on $S$, for the standard closure system $\< S,\phi\>$, via the concept of  minimal covers, can be recovered from the canonical basis. The main result is achieved in Theorem \ref{tr} by the series of Lemmas \ref{DG and D}-\ref{reduced}. This allows us to recognize effectively the systems without $D$-cycles that will be treated further in section \ref{Ebas}.

Recall that $bDa$, for $a,b \in S$ in a standard closure system $\< S, \phi\>$, iff there is a $\ll$-minimal cover $A$ for $b$ such that $a \in A$. In particular, the closure system satisfies the implication $A \rto b$. We note that such an implication belongs to the non-binary part of the $D$-basis of $\< S,\phi\>$, see \cite{ANR11}. Vice versa, if $A\rto b$ is a non-binary implication of the $D$-basis, then $A\cup \{b\}\subseteq \op{Ji} L$ and $A$ is a minimal cover for $b$. The binary part of $\Sigma_D$ can be given by $\succ_\phi$.

We will denote $D^\delta$ the dual of $D$, i.e., $(a,b) \in D^\delta$ iff $(b,a) \in D$. Thus, $a$ is the first entry in the pair $(a,b)\in D^\delta$ iff $a\in A$ appears on the left in a \emph{non-binary} implication $A\rto b$ from the $D$-basis.

Recall that in the standard closure system $\< S, \phi\>$, the binary relation $\geq_\phi$ on $S$ defined as:
\[ a\geq_\phi b \text{  iff  } b \in \phi(a)
\]
is a partial order.  A subset $J\subseteq S$ will be called a $\geq_\phi$-ideal, if $a \in J$ and $a\geq_\phi b$ imply $b \in J$. We will denote by $\Id_\phi(A)$ an $\geq_\phi$-ideal generated by $A\subseteq S$.

\begin{df}\label{delta}
Given any \emph{regular} basis $\Sigma$, we define a binary relation $\Delta_\Sigma$ on $S$ as follows: $(a,b)\in  \Delta_\Sigma$ iff there is $(A\rto B) \in \Sigma^{nb}$ such that $a \in A, b \in B$. By $\Delta_\Sigma^{tr}$ we will denote the transitive closure of the relation $\Delta_\Sigma$.
\end{df}

Finally, we introduce an important modification of the canonical basis $\Sigma_C$ that will prove useful.
For each $A_C\rto B_C$ in $\Sigma_C^{nb}$ let $A_K\rto B_K$ be a corresponding implication from any $K$-basis of the closure system. Replace each implication $A_C\rto B_C$ in $\Sigma_C^{nb}$ by $A_K\rto B_C$. Denote by $\Sigma^*$ this new set of implications which, apparently, also forms a basis. Note that $\Sigma^*$ is a refinement of the canonical basis, with no change to the right sides of implications.

We note that although $\Sigma_C$ is uniquely defined for a closure system, $\Sigma^*$ is not, reflected in the fact that there might be several $K$-bases for the closure system.

Our main goal now is to prove the following statement. 

\begin{thm}\label{tr}
$(D^\delta)^{tr}=\Delta_{\Sigma^*}^{tr}$
\end{thm} 

The meaning of Theorem \ref{tr} is that the basis $\Sigma^*$ recovers an important sub-relation $\Delta_{\Sigma^*}$ of the $D$-relation, so that any pair in the $D$-relation can be obtained via transitive closure of this sub-relation.

The theorem directly follows from Lemmas \ref{DG and D} and \ref{reduced}.

\begin{lem}\label{DG and D}
$\Delta_{\Sigma^*}\subseteq D^\delta$.
\end{lem}
\begin{proof}
We need to show that, for any $A_C\rto B_C$ from the canonical basis, if $A_K\subseteq A_C$ is a $\ll$-minimal representation of $U=\phi(A_C)$, then 
 $(a,b) \in D^\delta$, for every $a\in A_K$ and $b \in B_C$.
Apparently, $A_K$ is a non-trivial cover of $b$, for every $b \in B_C$. Suppose that $A_K$ can be properly refined to another cover $A''$ of $b$, i.e. $A'' \ll A_K$, $A''$ is a cover for $b$ and $A_K\not\subseteq A''$. Since $A_K$ is a $\ll$-minimal representation for $U=\phi(A_C)$, we cannot have $\phi(A'')=U$. Hence, $\phi(A'')\subset \phi(A_K)=U$. Then, since $A_C$ is a quasi-closed set and $A''\subseteq A_C$, we should have $b \in \phi(A'')\subseteq A_C$, a contradiction. Hence, $A_K$ is a minimal cover for $b \in B_C$ and $(a,b)\in D^\delta$ follows, for every $a\in A_K$. 
\end{proof}

This lemma allows us to establish a connection between any $K$-basis and the $D$-basis $\Sigma_D$.

\begin{cor}
For any $K$-basis $\Sigma_K$, its unit expansion $\Sigma_K^u$ is contained in $\Sigma_D$.
In particular, $|\Sigma_K^u|\leq |\Sigma_D|$.
\end{cor} 
\begin{proof}
Indeed, as shown in Lemma \ref{DG and D}, if $A_K\rto B_C$ is in $\Sigma^*$, then $A_K$ is a minimal cover for every $b \in B_C$. Recall that $A_K$ is taken from implication $A_K\rto B_K$ of some $K$-basis, and $B_K=\max(B_C)\subseteq B_C$. Hence every implication $A_K\rto b$, with $b\in B_K$, is in the $D$-basis.
\end{proof}

We observe that Theorem \ref{tr} does not hold when $\Sigma^*$ is replaced by $\Sigma_C$.
This is due to the fact that an even weaker form of Lemma \ref{DG and D} fails for $\Sigma_C$.

\begin{exm}\label{A12}
It is possible that $\Delta_{\Sigma_C}\not \subseteq (D^\delta)^{tr}$.
Consider the closure system $\<\{1,2,3,4,5,6\},\phi\>$ with $\op{Cl}(X,\phi)$ on Fig.~\ref{A12Fig}, whose canonical basis $\Sigma_C$ is
\[
2\rto 1,3\rto 1,5\rto4,6\rto 13, 14\rto 3, 123\rto 6, 1345\rto 6, 12346\rto 5.
\]

Then $\Sigma^{*nb}$ is 
\[
14\rto 3,23\rto 6, 15\rto 6, 24\rto 5
\]

and $\Sigma_D^{nb}$ is 
\[
14\rto 3,23\rto 6, 15\rto 6, 24\rto 5, 24\rto 6
\]
We observe that $(6,5) \not \in (D^\delta)^{tr}$, since $6$ does not occur on the left side of the implications in $\Sigma_D^{nb}$, while $(6,5)\in \Delta_{\Sigma_C}$ due to the implication $12346\rto 5$ in $\Sigma_C$.

\begin{figure}[htbp]\label{A12Fig}
\begin{center}
\includegraphics[height=2.5in,width=5.0in]{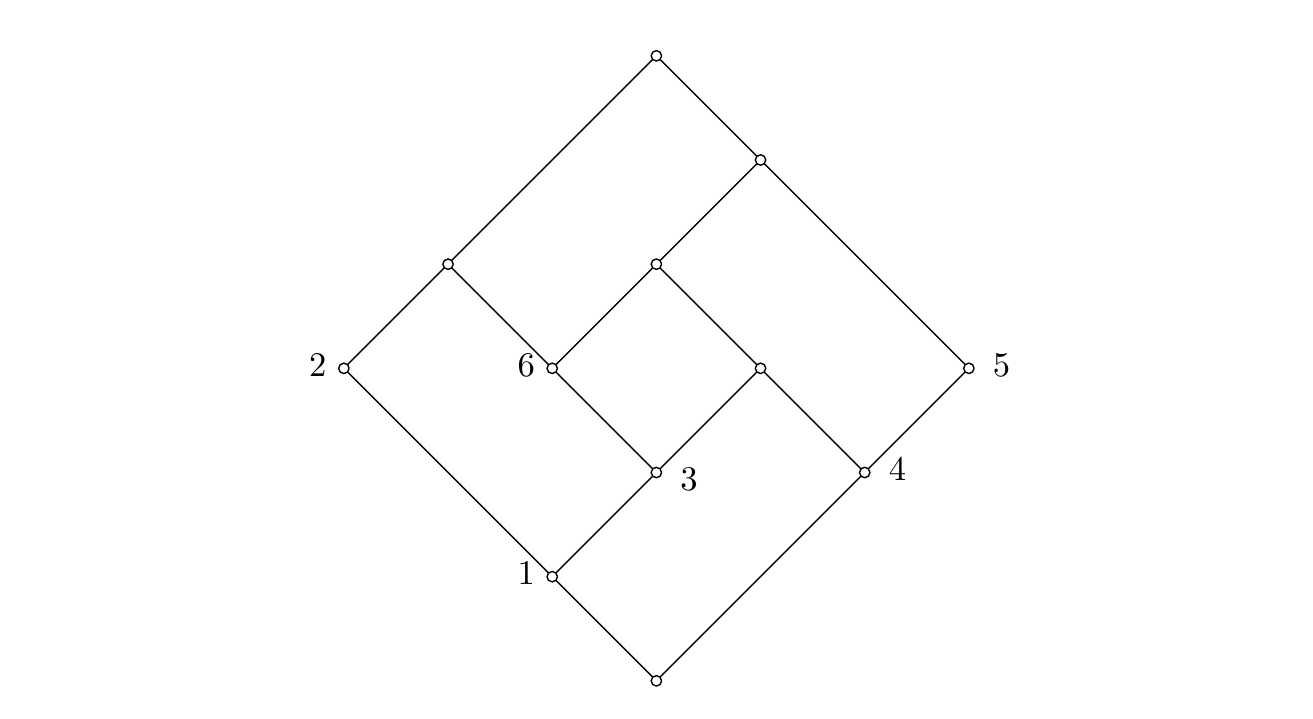}
\caption{Example \ref{A12}}
\end{center}
\end{figure}

\end{exm}

The reverse inclusion needed in Theorem \ref{tr} will come in Lemma \ref{reduced}, which will follow from a series of preliminary lemmas.

We will call a basis $\Sigma$ \emph{round}, if $A\cup B=\phi(A)$ for every $(A\rto B)\in \Sigma^{nb}$. In particular, due to the definition, the canonical basis is round.

\begin{lem}\label{prop123}
Let $\Sigma$ be any round basis of a closure system $\< S,\phi\>$. If an element $b\in S$ has a $\Sigma$-inference from $A\subseteq S$, then it has a $\Sigma^{nb}$-inference from $\Id_\phi A$. The converse statement is also true for an arbitrary basis $\Sigma$.
\end{lem}
\begin{proof}

Let $\sigma_k = A_k\rto B_k$, $k\leq m$, be an inference of $A\rto b$ from $\Sigma$. 

We want to show that one can remove binary implications from this sequence at the cost of replacing $A$ in properties (I) and (III) of Proposition \ref{inference} by $\Id_\phi A$. 

For this we show by induction on $k$ that set $X_k=(\Id_\phi A) \cup B_1 \cup \dots \cup B_{k-1}$ is an $\geq_\phi$-ideal. $X_1=\Id_\phi A$ is and ideal by the definition. Suppose $X_{k-1}$ is an ideal and $\sigma_k \in \Sigma^{nb}$.
Recall that  $A_k\cup B_k=\phi(A_k)$, a closed set in $\< S,\phi\>$, in particular, it is an $\geq_\phi$-ideal. As $A_k\subseteq X_{k-1}$, it follows that $X_k=X_{k-1}\cup B_k$ is an $\geq_\phi$-ideal as well. 

If $\sigma_k \in \Sigma^{b}$, then $A_k=\{a\}$ and $a \in X_{k-1}$, but then $B_k\subseteq X_{k-1}$ by the induction hypothesis.
In particular, we see from the above argument that $X_{k-1}=X_k$, when $\sigma_k\in \Sigma^b$, and any such $\sigma_k$ can be removed from the sequence.

The converse statement follows from the observation that $A\rto \Id_\phi A$ holds in the closure system, whence by Proposition \ref{inference}, there exists a $\Sigma$-inference of every element of $\Id_\phi A$ from $A$. One can append an $\Sigma^{nb}$-inference of $b$ from  $\Id_\phi A$, to obtain an inference of $b$ from $A$.
\end{proof}

\begin{lem}\label{min cov}
Let $\Sigma$ be any basis of a closure system $\< S,\phi\>$. Consider $aD^\delta b$, and let $A\rto b$ be the corresponding implication, where $A$ is a $\ll$-minimal cover for $b$ and $a \in A$. If $\sigma_1, \dots, \sigma_n$ is a $\Sigma^{nb}$-inference of $b$ from $\Id_\phi A$, then $a \in A_k$ for some $k\leq n$. 
\end{lem}
\begin{proof}
Since $A$ is a $\ll$-minimal cover for $b$, every $a\in A$ is a maximal element in $\Id_\phi A$. Therefore, $J_a=\Id_\phi A\setminus \{a\}$ is also an $\geq_\phi$-ideal. If some $a \in A$ does not appear in any $A_k$, $k\leq n$, then sequence of $\sigma_k$, $k\leq n$, is, in fact, a $\Sigma^{nb}$-inference of $b$ from $J_a$. 
Let $A'$ be the set of maximal elements of $J_a$.  Then $A'\ll A$ and $J_a=\Id_\phi A'$. By Lemma \ref{prop123}, there exists a $\Sigma$-inference of $b$ from $A'$, and thus $A'\rto b$ holds in the system. Since $A\not \subseteq A'$, this contradicts the $\ll$-minimality of cover $A$ for $b$. 
\end{proof} 

\begin{lem}\label{tran} 
For any regular basis $\Sigma$, let $\sigma_k \in \Sigma^{nb}$, $k\leq n$, be a non-redundant inference of $b$ from $\Id_\phi A$. Let $a \in A_s$, for some $s\leq n$. Then $(a,b)\in \Delta_{\Sigma}^{tr}$.  
\end{lem}
\begin{proof}
Proceed by induction on $t=n-k$. If $t=0$, i.e., $k=n$, then $a' \in A_n$, $b \in B_n$, whence $(a',b) \in \Delta_{\Sigma}$.
Suppose we have shown this for all $t<n-k$, and consider $\sigma_k=A_k\rto B_k$ and some $a'\in A_k$. If none of elements in $B_k$ belongs to any $A_s$, $s>k$, then this implication could be removed from the sequence, contradicting the assumption that the sequence is non-redundant. Hence, there is $b' \in B_k\cap A_s$, $s>k$. Since $n-s <n-k$, by the inductive assumption, $(b',b)\in \Delta_{\Sigma}^{tr}$. Besides, $(a',b') \in \Delta_{\Sigma}$. Hence, $(a',b) \in \Delta_{\Sigma}^{tr}$.
\end{proof}

Note that the following statement now follows from Lemmas \ref{prop123}-\ref{tran}.

\begin{cor}\label{D and DG}
Let $\Sigma$ be any regular round basis of a closure system $\< S,\phi\>$. Then $D^\delta \subseteq \Delta_{\Sigma}^{tr}$. In particular, this holds for $\Sigma=\Sigma_C$.
\end{cor}

It remains to show that the containment of Corollary \ref{D and DG} will still hold when we replace basis $\Sigma_C$ by $\Sigma^*$ defined prior to Theorem \ref{tr}.

\begin{lem}\label{reduced}
$D^\delta\subseteq \Delta_{\Sigma ^*}^{tr}$.
\end{lem} 

\begin{proof}
Consider $(a,b)\in D^\delta$. Let $A\rto b$ be an implication corresponding to a minimal cover $A$ of $b$, where $a \in A$. Let $\sigma_1,\dots, \sigma_n$ be a $\Sigma_C^{nb}$-inference of $b$ from $\Id_\phi A$. Replace $A_t$ in each $\sigma_t=A_t\rto B_t$, $t\leq n$, by the corresponding $\sigma^*=A_t^* \rto B_t$ in $\Sigma^*$.
Since $A_t^*\subseteq A_t$, and conclusions of all implications are still the same, the sequence $\sigma_1^*,\dots, \sigma_n^*$ is an $\Sigma^{*nb}$-inference of $b$ from $\Id_\phi A$. Remove unnecessary implications, if needed, making it non-redundant. 
By Lemma \ref{min cov}, $a \in A_k$ for some $k \leq n$, and by Lemma \ref{tran}, $(a,b)\in  \Delta_{\Sigma ^*}^{tr}$.
\end{proof}

Note that Lemma \ref{reduced} does not hold when $\Delta_{\Sigma ^*}^{tr}$ is replaced by $\Delta_{\Sigma ^*}$.

\begin{exm}
In the closure system of Example \ref{A12}, we have $(4,6)\in D^\delta$, but $(4,6)\not \in \Delta_{\Sigma^*}$. 
\end{exm}

The consequence is that the $D$-relation can be fully recovered from the transitive closure of the modified canonical basis $\Sigma^*$, whose premises are taken from any $K$-basis.

\begin{cor}
A closure system is without $D$-cycles iff $\Delta^{tr}_{\Sigma^*}$ does not have cycles.
\end{cor}
This follows from Theorem \ref{tr}.

\begin{thm}\label{mainD}
There exists a polynomial algorithm, in the size of $s(\Sigma_C)$, to recognize a closure system without $D$-cycles. 
\end{thm}
\begin{proof}
By Proposition \ref{Ktime}, it takes polynomial time to recover $\Delta_{\Sigma^*}$. Then it takes polynomial time to check whether this relation contains cycles.
\end{proof}
  
Note that the systems without $D$-cycles include, as a subclass, so-called \emph{quasi-acyclic} closure systems that are studied in the form of Horn Boolean functions in P.~Hammer and A.~Kogan \cite{HK}. For every basis $\Sigma$ one can define binary relation $\Box_\Sigma$ similar to $\Delta_\Sigma$: $(a,b)\in\Box_\Sigma$ iff there exists $(A\rto B)\in \Sigma$ such that $a\in A, b\in B$. The difference is that $\Delta_\Sigma$ is built only on the non-binary part of (regular) $\Sigma$. A quasi-acyclic closure system can be defined as a closure system for which there exists a basis $\Sigma$ such that $\Box_\Sigma^{tr}$ does not have cycles, unless a cycle is fully within the binary part $\Box_{\Sigma^b}^{tr}$. The latter cycles never occur in the standard closure systems considered in the current paper. In particular, if a closure system is quasi-acyclic, then it does not have $D$-cycles. See K.~Adaricheva \cite{A12} for further discussion and examples.

\section{$K$-basis in closure systems with unique criticals}\label{Kbas SD}

Starting with this section, we will consider the closure systems which we call $UC$-systems.

\begin{df}\label{UCsys} A closure system $\< S, \phi\>$ is called a system with \emph{unique criticals}, or $UC$-system, if, for any critical sets $C_1,C_2 \in \mathcal{C}$, $\phi(C_1)=\phi(C_2)$ implies $C_1=C_2$.
\end{df}

An immediate consequence of this definition is that, for every essential element $Y=\phi(Y)$, there exists a unique implication $C\rto \phi(C)\setminus C$ in the canonical basis, for which $\phi(C)=Y$.

\begin{prop} For every closure system $\< S, \phi\>$ given by some basis $\Sigma$, it requires time $O(s(\Sigma)^2\cdot |\Sigma|)$ to determine whether it is a $UC$-system.
\end{prop}
\begin{proof}
Indeed, it will require $O(s(\Sigma)^2)$ time to build the canonical basis. With $\Sigma_C$ available, one needs to check whether $\phi(C_1)\not = \phi(C_2)$, for any two premises $C_1,C_2$ of the canonical basis. It takes time $O(s(\Sigma_C))$ to find the closure of any given set, so verifying all the pairs requires time $O(s(\Sigma_C)\cdot n^2)$. Thus, overall time should not exceed $O(s(\Sigma)^2\cdot |\Sigma|)$.
\end{proof}

One important advantage of $UC$-closure systems is that one can establish a lower bound for $s_R(\Sigma^{nb})$, among all minimum bases $\Sigma$. 

Let $\mathcal{C}_{>1}\subseteq  \mathcal{C}$ be the subset of all critical sets with more than one element. For every $C \in \mathcal{C}_{>1}$, define $M(C) = \{y \in X: y \text{ is a } \geq_\phi-\text{maximal in } \phi(C)\setminus C \text{, and } y\not \in \phi(C') \text{ when } \phi(C') \subset \phi(C), C'\in \mathcal{C}_{>1}\}$.
 
Recall from Theorem~\ref{DG} (2) that any minimum basis of a closure system has the form $\Sigma=\{X_C\rto Y_C: C \in \mathcal{C}, \sigma(X_C)=C\}$.

\begin{lem}\label{max} Let $\<S,\phi\>$ be a $UC$-closure system, and $\Sigma$ be any minimum basis. 
If $(X_C\rto Y_C)\in \Sigma$ is an implication in $\Sigma$ corresponding to $C \in \mathcal{C}$, then
$M(C) \subseteq Y_C$.
\end{lem}
\begin{proof}
We have $y_m\in \phi(X_C)$, so the implication $X_C\rto y_m$ follows from $\Sigma$. Let $\sigma_1,\dots, \sigma_k \in \Sigma$, $\sigma_i=A_i\rto B_i$, be a $\Sigma$-inference of $y_m$ from $X_C$. By inductive argument, one can show that $A_i, B_i\subseteq \phi(X_C)$. Since $y_m$ is maximal in $\phi(X_C)\setminus C=\phi(C)\setminus C$, the last implication $\sigma_k$ is not binary. If $\phi(A_k) \subset \phi(C)$, then $A_k\subseteq C'$ for a critical set $C'\not = C$,  $\phi(C')=\phi(A_k)\subset \phi(C)$. Thus, $y_m$ is also the maximal in $\phi(C')\setminus C'$, which is not possible by assumption. Hence, $\phi(A_k)=\phi(X_C)=\phi(C)$. But in closure systems without $D$-cycles, there exists only one critical set for an essential element $\phi(C)$, i.e., $X_C=A_k$ and $y_m\in Y_C$.
\end{proof}

\begin{cor}\label{sR} If $\mathcal{C}_{>1}=\{C_1,\dots, C_k\}$, then, for every minimum basis  $\Sigma$ of an $UC$-system, $s_R(\Sigma^{nb})\geq |M(C_1)|+\dots + |M(C_k)|$.
\end{cor}

We will see in section \ref{optE} that this lower bound can be attained in the $E$-basis of a closure system without $D$-cycles.  
\vspace{0.3cm}

One important subclass of $UC$-systems is the class of closure systems whose closure lattice satifies $(SD_\vee)$.
We will call an element $t \in \op{Cl}(S,\phi)$ \emph{join-semidistributive}, if $(SD_\vee)$ holds at $t$: $t=x\vee y =x\vee z \rto t=x\vee (y\wedge z)$.
We first recall a well-known theorem of  B. J\'onsson and J.E. Kiefer.

\begin{thm}\label{JK62}\cite{JK62}
An element $t$ of a finite lattice is \jsd\! iff $t$ has a unique $\ll$-minimal representation.  
\end{thm} 

We will call closure system $\< S, \phi\>$ \emph{join-semidistributive} if all elements of $\op{Cl}(S, \phi)$ are such. 
As was pointed in section \ref{kbas},  the premises of implications in $K$-basis, in its non-binary part, are associated with the $\ll$-minimal join-representation of essential elements in the closure lattice. 

\begin{prop}\label{UC} Let standard closure system $\< S, \phi\>$ be \jsd\!. Then
\begin{itemize}
\item[(1)] $\< S, \phi\>$ is an $UC$-closure system;
\item[(2)] $\< S, \phi\>$ has unique $K$-basis.
\end{itemize}
\end{prop}
\begin{proof}
(1) Fix any essential set $Y$, and let $C$ be any critical subset with $\phi(C)=Y$. Then $C$ should contain some minimal order generator $C_K$, and $\sigma(C_K)=C$. On the other hand, $C_K$ gives a $\ll$-minimal representation of $\phi(C)=Y$, and this representation is unique in closure systems with $(SD_\vee)$. Hence, $C=\sigma(C_K)$ is unique critical set associated with the essential element $Y$.

(2) Since, by Theorem \ref{JK62}, every element of $\op{Cl}(S, \phi)$ has a unique $\ll$-minimal join representation,
apply Proposition \ref{12} to conclude that the $K$-basis of the associated standard closure system must be unique. 
\end{proof}

For both statements in Proposition \ref{UC}, $(SD_\vee)$ is a sufficient condition. It is easy to show that it is not necessary. Example \ref{2Kbases} provides an $UC$-system which has two $K$-bases, and hence is not \jsd\!.
Even if the closure system is both $UC$ and has a unique $K$-basis, which guarantees that $(SD_\vee)$ holds at all essential elements, it does not imply that $(SD_\vee)$ holds at non-essential elements. 

\begin{exm}\label{SD+fails}
Take a closure system on $S=\{a,b,c,d\}$ defined by implications $\Sigma_C=\{ac \rto b, bd\rto c\}$. There are two essential elements $\phi (ac),\phi (bd)$, and $SD_\vee$ holds at both. In particular, the $K$-basis is unique and coincides with $\Sigma_C$. Nevertheless,  $(SD_\vee)$ fails at the non-essential element $\phi(abcd)=1$: $(b\vee c)\vee a=(b\vee c)\vee d=1$, but $(b\vee c)\vee (a\wedge d)=b\vee c$.\\
\end{exm}

This leaves us with the following open questions.

\begin{problem}
\end{problem}
\begin{itemize}
\item[(A)] Is there an algorithm that allows us to recognize in polynomial time whether the system has a unique $K$-basis?
\item[(B)] Is there an effective algorithm that allows us to recognize whether a closure system $\< S, \phi\>$ is join-semidistributive, given its canonical basis?
\end{itemize}
\vspace{0.2in}

A positive answer to question (A), even in $UC$-system, does not automatically imply the positive answer to (B), as Example \ref{SD+fails} demonstrates. 
We will address lattice theoretical aspects of $UC$-systems in a separate paper \cite{AN12}.



At the end of this section, we discuss a new binary part for \jsd\ closure systems.

Recall from Lemma \ref{bin opt} that any implication $(x\rto A)\in \Sigma^b$ can be replaced by $x\rto B$ with  $\phi(B)=\phi(\{x\})\setminus \{x\}$.

In the most closure systems we treated so far, the binary part was built out of the cover relation of the order $\geq_\phi$.
Precisely, $B=\max_{\geq_\phi} (\phi(\{x\})\setminus \{x\})$, for every $(x\rto B)\in \Sigma^b$. In \jsd\ closure systems, there is another natural choice for $B$.

\begin{df} We will call set of implications $\Sigma_F$ the $F$-\emph{basis} of a \jsd\ standard closure system, if 
\begin{itemize}
\item[(1)] $\Sigma_F^{nb}=\Sigma_K^{nb}$;
\item[(2)]$\Sigma_F^b=\{x\rto B: B \text{ is a minimal \emph{order} generator of } \phi(\{x\})\setminus \{x\}\}$.
\end{itemize}
\end{df}

We note that $F$-basis of any \jsd\ closure system is unique. We saw in Example \ref{cover} that binary part based on the cover relation $\succ_\phi$ might be redundant. The advantage of the binary part of the $F$-basis is that it possesses a weak form of minimality. Recall that $\Sigma_u^b$ denotes the unit expansion of the implications $\Sigma^b$.

\begin{prop}\label{irr}
 $\Sigma^b_u$ of the $F$-basis $\Sigma_F$ for a  \jsd\ closure system $\<S, \phi\>$ is non-redundant. 
\end{prop}
\begin{proof}
We need to show that none of the implications in $\Sigma^b_u$ of the $F$-basis can be removed. 

Suppose that $a\rto B$ is a binary implication in $\Sigma_F$, where $B=\{b_0,\dots, b_n\}$ is a minimal order generator of $\phi(\{a\})\setminus \{a\}\}$.
Consider a closure system $\Sigma =\Sigma_F\setminus \{a\rto b_0\}$. Let $\phi_\Sigma$ be an associated closure operator.
\begin{claim}
$\phi_\Sigma(\{a\})=\phi (\{b_1,\dots, b_n\})\cup\{a\}$.
\end{claim}
Indeed, to show $\supseteq$, we note that $b_1,\dots, b_n, a \in \phi_\Sigma(\{a\})$. All the implications of $\Sigma_F$ that can be applied to elements of $\phi (\{b_1,\dots, b_n\})$ are also in $\Sigma$, due to the fact that $a \not \in \phi (\{b_1,\dots, b_n\})$.

Now take any $w \in \phi_\Sigma(\{a\})$, with $w \neq a$. Let $\sigma_1,\dots,\sigma_k \in \Sigma$ be a $\Sigma$-inference of $w$ from $\{a\}$. We may assume that $\sigma_1=(a\rto \{b_1,\dots,b_n\}$ and $a$ does not appear in a premise of any implication $\sigma_2,\dots, \sigma_k$, since $\Sigma_F$ is regular. Hence, $A_i,B_i \subseteq \phi(\{b_1,\dots, b_n\})$, for all $i\geq 2$, by the inductive argument.  This implies $w \in \phi (\{b_1,\dots, b_n\})$, which proves the Claim.

It immediately follows from the Claim that $\phi_\Sigma(\{a\}) \subset \phi(\{a\})$, since $b_0 \not \in \phi (\{b_1,\dots, b_n\})$, due to the $\ll$-minimality of $B$ for $\phi(\{a\})\setminus \{a\}\}$.
\end{proof}

We note that there is no guarantee that $|B|$ in an implication $(a\rto B)\in \Sigma_F$ for some \jsd\ closure system is minimal with the property $\phi(B)=\phi(\{a\})\setminus \{a\}\}$.  Examples of such systems will be shown in section \ref{NP}.

\section{$E$-basis and its connection to the canonical basis}\label{Ebas}

We introduced the $E$-basis in \cite{ANR11} as a finer version of the $D$-basis for closure systems without $D$-cycles.
In section \ref{D-rel} we showed that one can recognize whether a closure system is without $D$-cycles from its canonical basis, in time polynomial of the size of that basis. 

In this section, we will establish the connection between the $E$-basis and the canonical basis. 
Since every finite lattice without $D$-cycles is join-semidistributive, the results of section \ref{Kbas SD} are applicable here. In particular, the $K$-basis is unique for closure systems without $D$-cycles.

We will proceed as follows. First, we prove that the $E$-basis in its aggregated form is a refinement of the canonical basis. Secondly, we establish that the left sides of the $E$-basis and the $K$-basis are the same. Thirdly, we will analyze possible differences in the right sides of the non-binary parts of the $K$-basis and the $E$-basis. This will allow us to create a blend of the two, which we will call the optimized $E$-basis. There is an effective algorithm to produce the optimized $E$-basis from the canonical basis. 

First, recall the definition of the $E$-basis from \cite{ANR11}.

Let $\< S, \phi\>$ be a standard closure system without $D$-cycles.
For every $x \in S$, let $M(x)=\{Y\subseteq S: Y \text{ is a minimal cover of } x\}$. 
The family $\phi(M(x))=\{\phi(Y): Y \in M(x)\}$ is ordered by set containment, so we can consider its $\subseteq$-minimal elements.
Let $M^*(x)=\{Y \in M(x): \phi(Y) \text{ is $\subseteq$-minimal in } \phi(M(x))\}$.

The set of implications $\Sigma_E=\Sigma_E^b \cup \Sigma_E^n$:
\begin{itemize}
\item[(1)] $\Sigma_E^b= \{y \rightarrow x : y\succ_\phi x\}$,
\item[(2)] $\Sigma_E^{nb}=\{X \rightarrow x: X \in M^*(x)\}$
\end{itemize} 
is called the $E$-\emph{basis} of the closure system. 
Recalling the definition of the $D$-basis given in section \ref{D-rel}, we note that $\Sigma_E\subseteq \Sigma_D$.

It was noted in \cite{ANR11} that $\Sigma_E$ might not form a basis of a closure system in general, so the requirement of no $D$-cycles is essential.

\begin{exm}\label{Co4}
Consider a closure system given by the basis $\Sigma=\{ac\rto b, bd\rto c, ad\rto bc\}$, which is the canonical basis, at the same time the $K$-basis of the system. It is easy to check that this system is \jsd\!, but it has the $D$-cycle $bDcDb$. We have $\{a,c\}, \{a,d\} \in M(b)$, and $\phi(\{a,c\})\subset \phi(\{a,d\})$, so that only $\{a,c\}$ belongs to $M^*(b)$. Symmetrically, only $\{b,d\}$ belongs to $M^*(c)$. Hence, the set of implications $\Sigma_E$ for this system is $\{ ac \rto b, bd\rto c\}$, which is not a basis, since $\{a,d\}$ would be a closed set with respect to $\Sigma_E$. 
\end{exm}

Let\/ $\Sigma_E^{ag}$ be the aggregated $E$-basis and $\Sigma_C$ be the canonical basis of $\< S, \phi\>$. 
As usual, $\sigma$ denotes the saturation operator associated with the closure operator $\phi$.
\begin{thm}\label{mainE}
Let $\langle S, \phi\rangle$ be a standard closure system without $D$-cycles. Then there is a one-to-one mapping $f:\Sigma_E^{ag}\longrightarrow \Sigma_C$ such that $f(A\rto A')=(\sigma(A)\rto \phi(A))$ with $A'\subseteq \phi(A)\setminus \sigma(A)$. In particular, the aggregated $E$-basis is a refinement of $\Sigma_C$: $s(\Sigma_E^{ag})\leq s(\Sigma)$.
\end{thm}

In fact, the result follows, if we argue over $\Sigma_E$ in its original unit form, and the unit expansion $\Sigma_u$ of $\Sigma_C$, and show that there is one-to-one mapping $f^*:\Sigma_E\longrightarrow \Sigma_u$ such that $f^*(A\rto x)=(\sigma(A)\rto x)$. 

We now describe a well-known algorithm of producing the canonical basis due to A. Day~\cite{D92}.


Given any basis $\Sigma=\{A_i\rto B_i: i\leq k\}$ of a closure system $\langle S, \phi\rangle$, the procedure of obtaining the canonical basis $\Sigma_C$ consists of two steps:
\begin{itemize}
\item[(1)] replace each $A_i\rto B_i$ by $\sigma(A_i)\rto \phi(A_i)$;
\item[(2)] if there are two implications $C\rto F$ and $D\rto F$ obtained in the first step, and $C\subseteq D$, then remove $D\rto F$; also, remove all implications of the form $F\rto F$.
\end{itemize}

The unit expansion $\Sigma_u$ of $\Sigma_C$ can now be written from all remaining implications and will consists of $\sigma(A_i)\rto x$, where $A_i$ is some premise in the original basis and $x \in \phi(A_i)\setminus\sigma(A_i)$.

In the proof of Theorem \ref{mainE} below, we will apply this algorithm to a given basis $\Sigma_E$ of some closure system without $D$-cycles. First, we prove several auxiliary statements. We need to recall that the saturation operator $\sigma$ associated with a closure operator $\phi$ is defined as $\sigma(X)=\bigcup_{k\geq 1} q^k(X)$, where 
$q(X)= X\cup \bigcup\{\phi(Y): Y\subseteq X \text{ and } \phi(Y)\subset \phi(X)\}$.

\begin{lem}\label{saturation}
If $(A\rto x)\in \Sigma^{nb}_E$, then $x \in \phi(A)\setminus \sigma(A)$.
\end{lem}
\begin{proof} It is enough to show that $x \not \in \sigma(A)$. 
We have $x \not \in A$. Also, $x\not \in \phi(Y)$, for any $Y\subseteq A$ with $\phi(Y)\subset \phi(A)$. Indeed, otherwise, $Y$ would be a cover for $x$, so we would be able to find a minimal cover $Y'$ with $Y'\ll Y$, hence, $\phi(Y') \subset \phi(A)$. This implies that $A\not \in M^*(x)$, hence, $A\rto x$ cannot be included in $E$-basis. Thus, $x \not \in q(A)$. Apparently, $x\not \in q(A)$ with $\phi(q(A))=\phi(A)$, so we can apply the same argument, replacing $A$ by $q(A)$, to show that $x \not \in q^2(A)$. Proceeding with this inductive argument, we conclude that $x \not \in q^k(A)$, for all $x\geq 1$.
\end{proof}

\begin{lem}\label{one} 
If $(A\rto x)\in \Sigma_E$, then $A$ is a $\ll$-minimal join representation of $\phi(A)$.
\end{lem}
\begin{proof}
We may assume that $|A|>1$, since for binary implication the statement is trivial.

According to the definition of the $E$-basis, $A$ is a minimal cover of $x$. If there is a set $F\ll A$ such that $\phi(F)=\phi(A)$ and $A\not\subseteq F$, then $F$ will also be a cover of $x$, which contradicts the $\ll$-minimality of the cover $A$.

\end{proof}

\emph{Proof of Theorem.}
Apply the algorithm for obtaining the canonical basis to $\Sigma_E$. 
Apparently, $\Sigma_E^b\subseteq \Sigma_u^b$.
So we need to make sure the one-to-one mapping $f^*$ exists for non-binary implications.

According to Lemma \ref{one}, if $(A\rto x),(B\rto y)\in \Sigma_E$ and $A\not = B$, then $\phi(A)\not =\phi(B)$.
Thus, after applying the first step of the algorithm, we will not obtain any two implications $C\rto F$, $D\rto F$, with $C \neq D$.
Moreover, by Lemma \ref{saturation}, we will not have implications of the form $F\rto F$ after the first step.
Hence, the second step of the algorithm may be applied only to remove repeating implications, and they may occur here only when $A=B$. In this case $x\not = y$, and $x,y\in \phi(A)\setminus \sigma(A)$ by Lemma \ref{saturation}.
Therefore, $\sigma(A) \rto x$ and $\sigma(A)\rto y$ are in the unit expansion of the implication $\sigma(A)\rto\phi(A)$ from the canonical basis. Hence, the mapping $f^*(A\rto x) = (\sigma(A)\rto x)$ from $\Sigma_E$ to $\Sigma_u$ is one-to-one.
\emph{End of Proof.}\\

Now we show that the premises of implications in the $E$-basis and $K$-basis for a closure system without $D$-cycles are the same.

\begin{lem}
Let $\< S, \phi\>$ be a standard closure system without $D$-cycles with $E$-basis $\Sigma_E$ and $K$-basis $\Sigma_K$. For each critical set $C$, if $(C_E\rto Y_E)\in \Sigma_E^{ag}$ and $(C_K\rto Y_K)\in \Sigma_K$ with $\sigma(C_E)=\sigma(C_K)=C$, then $C_E=C_K$.
\end{lem}
\begin{proof}
According to Proposition \ref{12}, $C_K$ is a $\ll$-minimal join representation of $\phi(C)$, which is unique is \jsd\ closure systems. Due to Lemma \ref{one}, $C_E=C_K$.
\end{proof}

The following examples show that the conclusions of non-binary parts of the $E$-basis and $K$-basis might be different.

\begin{exm}\label{EO}
The $K$-basis may have smaller size than the $E$-basis.
\end{exm}
Consider the closure system on $S=\{a,b,c,d\}$ given by its canonical basis $\Sigma_C=\{d\rto cb, c\rto b, ab\rto dc\}$. This is also the $E$-basis. On the other hand, the $K$-basis optimizes the last implication to $ab\rto d$, since $d$ is the maximal element of $\{d,c\}$ with respect to the $\geq_\phi$-order.

\begin{exm}
The $E$-basis may have smaller size than the $K$-basis.
\end{exm}
Consider the closure system on $S=\{2,3,4,5\}$ defined by its canonical basis $\Sigma_C=\{2\rto 5,45\rto 23, 35\rto 2\}$. This is also the $K$-basis of the system. On the other hand, $2$ appears on the right side of two non-binary implications, in particular, both $\{4,5\}$ and $\{3,5\}$ are the minimal covers for $2$. Since $\phi(45)>\phi(35)$, only $35\rto 2$ appears in the $E$-basis. Thus, the $E$-basis has smaller size than the $K$-basis: $\Sigma_E=\{2\rto 5,45\rto 3, 35\rto 2\}$.

\begin{df}
We call $\Sigma_{OE}$ an optimized $E$-basis if every implication $A_E\rto B_E$ from the non-binary part of $E$-basis is replaced by $A_E\rto B_{OE}$, where $B_{OE}=\max_{\geq_\phi}(B_E)$.  
\end{df}

Note that the $E$-basis is \emph{ordered direct}, as was shown in Theorem 23 \cite{ANR11}. The optimized $E$-basis might not longer have this property. On the other hand, applying the ordered sequence of the $E$-basis to the optimized $E$-basis and concatenating the binary part again at the end will produce an ordered sequence. See the definition of the ordered sequence and further details in section 8 of \cite{ANR11}.\\

Let us demonstrate this on Example \ref{EO}. The optimized $E$-basis is the same as the $K$-basis: $\Sigma_{OE}=\{d\rto c, c\rto b, ab\rto d\}$. $\Sigma_{OE}$ is no longer direct, with this order of implications that is inherited from the order of the $E$-basis, since $\rho(ab)=abd\not = abcd= \phi(ab)$.
If we concatenate the binary part at the end, we obtain the following sequence:  $d\rto c, c\rto b, ab\rto d, d\rto c, c\rto b$. It is easy to verify that it is an ordered direct sequence. 

\begin{lem}
Given a closure system without $D$-cycles, it will take time $O(s^2(\Sigma_C))$ to obtain optimized $E$-basis from the canonical basis.
\end{lem}
\begin{proof}
Given the canonical basis, it takes time $O(s^2(\Sigma_C))$ to verify that there is no $D$-cycles and to obtain the unique $K$-basis $\Sigma_K$. For each $x \in S$, choose all non-binary implications $(X_K\rto Y_K)\in \Sigma_K$ such that $x \in Y_K$. We know that $X_K$ is a $\ll$-minimal cover for $x$. Compute $\phi(X_K)$ for all such implications and choose minimal ones, with the respect to the containment order. Keep $x$ in $Y_K$ only if $\phi(X_K)$ is $\subseteq$-minimal; otherwise, remove $x$ from $Y_K$. There are no more than $|X|$ elements to check, and each may appear in no more than $|\Sigma_C|$ implications. Computation of $\phi(X_K)$ takes time $O(s(\Sigma_C))$. Assuming that $|X|\cdot |\Sigma_C|\approx s(\Sigma_C)$, it will take time $O(s^2(\Sigma_C))$ to build the optimized $E$-basis.
\end{proof}

\section{Aspects of optimality of the $E$-basis and its further modifications}\label{optE}

As was discussed in sections \ref{min-opt} and \ref{bin}, every optimum basis of a standard closure system has  parameters fully determined by the closure operator: the right size of each binary implication (associated with one-element critical sets) and the left size of each non-binary implication (associated with critical sets of more than one element).

In section \ref{NP} we will show that finding any of those optimal parts in closure systems without $D$-cycles is an NP-complete problem.  
On the other hand, the optimum right size of the non-binary part can be achieved. 

First, we summarize from the previous results.

\begin{prop}\label{total right} Let $\<S, \phi\>$ be a standard closure system. There exists a constant $s_R^{nb}(S,\phi)$ such that $s_R(\Sigma^{nb})=s_R^{nb}(S,\phi)$ for every optimum basis $\Sigma$. 
\end{prop}
\begin{proof} 
Indeed, $s(\Sigma)=s_L(\Sigma^b)+s_R(\Sigma^b)+s_L(\Sigma^{nb})+s_R(\Sigma^{nb})$, for every basis $\Sigma$. Now assume that $\Sigma$ is optimum. The number of implications in $\Sigma^b$ is fully determined by the number of one-element critical sets, so that $s_L(\Sigma^b)$ is the same for all optimum bases. By Theorem \ref{rs-bin}, $s_R(\Sigma^b)$ is the sum of parameters $b_C$, where $C$ runs over all singletons in $\mathcal{C}$, and $s_L(\Sigma^n)$ is the sum of parameters $k_C$, where $C$ runs over all non-singletons in $\mathcal{C}$, as shown in Theorem \ref{DG} (2). Therefore, the last components in the sum, $s_R(\Sigma^{nb})$, must be the same for all optimum bases.
\end{proof}


\begin{cor}\label{sr}
The inequality $s_R^{nb}(S, \phi) \leq s_R(\Sigma_*^{nb})$ holds for any basis $\Sigma_*$ of a standard closure system $\<S, \phi\>$. 
\end{cor}
\begin{proof}
By Lemma \ref{bin opt} we may assume that the binary part of $\Sigma_*$ matches the binary part of some optimum basis $\Sigma$. Now $\Sigma$ is right-side optimum, due to Corollary \ref{lr}, whence $s_R(\Sigma)\leq s_R(\Sigma_*)$.
Since these bases have identical binary parts, we also have $s_R^{nb}(S, \phi)=s_R(\Sigma^{nb})\leq s_R(\Sigma_*^{nb})$.
\end{proof}

\begin{rem}
Unlike the parameters $b_C$ and $k_C$ of parts of optimum basis used in argument of Proposition \ref{total right}, there is no fixed value for the size of the \emph{conclusion} in the \emph{non-binary} implication $A_C\rto B_C$ of any optimum basis.  
\end{rem}

\begin{exm}\label{ex66} Let the closure system on $S=\{a,b,c,z\}$ be given by the canonical basis $\Sigma_C=\{z\rto a, ab\rto cz, ac\rto bz\}$. There are two optimum bases: $\Sigma_1=\{z\rto a, ab\rto c, ac\rto bz\}$ and $\Sigma_2=\{z\rto a, ab\rto cz, ac\rto b\}$. While $s_R(\Sigma_1^{nb})=s_R(\Sigma_2^{nb})=3$, implications in $\Sigma_1$ and $\Sigma_2$ with premise $ab$ have conclusions of different sizes, and the same holds for implications with $ac$ in the premise.
\end{exm}

On the other hand, we notice that the closure system in Example \ref{ex66} is not $UC$, since essential element $\phi(ab)=\phi(ac)$ contains two critical subsets.

\begin{conj}  If $A_1\rto B_1,\dots, A_k\rto B_k$ are \emph{all} implications of any optimum basis, corresponding to critical sets with the same closure, i.e., $\phi(A_1)=\dots =\phi(A_k)$, then $s=|B_1|+\dots +|B_k|$ does not depend on the choice of the optimum basis. In particular, in $UC$-closure systems, for every critical set $C$, and corresponding implication $A_C\rto B_C$ in any optimum basis, $|B_C|$ does not depend on the choice of the optimum basis. 
\end{conj}

The main statement of this section is that the optimized $E$-basis of any closure system without $D$-cycles has the optimum right size in its non-binary part. For this, we recall that every closure system without $D$-cycles is \jsd\!, and thus it is also a $UC$-system.

\begin{thm}\label{rs-min}
Let $\<S,\phi\>$ be any closure system without $D$-cycles. Then $s_R(\Sigma_{OE}^{nb})= s_R^{nb}(S,\phi)$.
\end{thm}
\begin{proof}
Consider any optimum basis $\Sigma$. Since it is minimum, Lemma \ref{max} is applicable to $\Sigma$. Take any $(X_C\rto Y_C)\in \Sigma^{nb}$ and corresponding $(A_C\rto B_C)\in \Sigma_{OE}^{nb}$, $C \in \mathcal{C}$.
By the definition of $\Sigma_{OE}$, $B_C$ consists of maximal elements $y_m$ in $\phi(A_C)\setminus \sigma(A_C)=\phi(C)\setminus C$ such that $y_m \not \in \phi(C')\setminus C'$, for every $C'\in \mathcal{C}$ with $\phi(C')\subset \phi(C)$.
According to Lemma \ref{max}, this implies $B_C\subseteq Y_C$, hence, $s_R(\Sigma_{OE}^{nb})\leq s_R(\Sigma^{nb})$.
On the other hand,  $s_R(\Sigma_{OE}^{nb}) \geq s_R^{nb}(S, \phi)=s_R(\Sigma^{nb})$, by Corollary \ref{sr}.
Thus, $s_R(\Sigma_{OE}^{nb})= s_R^{nb}(S,\phi)$.
\end{proof}

\begin{rem}
We note that the result of Theorem \ref{rs-min} cannot be extended to \jsd\ closure systems, when replacing the $E$-basis by a $K$-basis. We saw in Example \ref{Co4} that the conclusions of non-binary implications in a $K$-basis cannot be reduced by taking away some $y_m \in \phi(C)\setminus C$, even if $y_m \in \phi(C')\setminus C'$, for some critical sets $C,C'$, $\phi(C')\subset \phi(C)$. The same example shows that a $K$-basis may not reach $s_R^{nb}(S,\phi)$ in the size of its non-binary right side. Indeed, the implication $ad\rto bc$ can be reduced to either $ad\rto b$ or $ad\rto c$, to obtain a right-side optimum basis.
\end{rem}

At the end of this section we mention the binary modification for basis $\Sigma_{OE}$ that is available for all \jsd\ closure systems discussed in section \ref{Kbas SD}. It combines the binary part of the $F$-basis and non-binary part of the optimized $E$-basis.

\begin{df} 
$\Sigma_{FOE}=\Sigma_F^b\cup \Sigma_{OE}^{nb}$.
\end{df}

\begin{prop}
The unit expansion of $\Sigma_{FOE}$ is non-redundant.
\end{prop}
Indeed, this follows from Theorem \ref{rs-min} and Proposition \ref{irr}.

\section{Finding optimum or minimum unit basis for bounded lattices is NP-complete}\label{NP}

As was noted in section \ref{min-opt}, finding an optimum basis for an arbitrary closure system is an NP-complete problem, while it could be done effectively in closure systems with modular closure lattices.

The effective procedure that allows us to build the optimized $E$-basis for systems without $D$-cycles and reaches the optimum in its non-binary right side, could suggest that such systems also have an effective procedure for finding an optimum basis.  The goal of this section is to show, that, to the contrary, finding an optimum basis even in the strict subclass of systems without $D$-cycles (closure systems with \emph{bounded} closure lattices) is NP-complete.

In this section we show that each of the problems is NP-complete, for the closure systems without $D$-cycles: 
\begin{itemize}
\item[(1)] find a basis $\Sigma$ that reaches the minimum in $s_L(\Sigma^{nb})$;
\item[(2)] find a basis $\Sigma$ that reaches the minimum in $s_R(\Sigma^b)$.
\end{itemize}

For this aim, we need to identify a known NP-complete problem that can be reduced to either of the problems above. 

The following problem known as \emph{the set cover problem} is included into original ``Karp's21", the list of 21 NP-complete problems in Karp \cite{K72}.


Given a finite set $Q$ and a family of its subsets $\mathcal{Q}=\{Q_i:i\leq k\}$ that \emph{covers} $Q$, i.e., $Q\subseteq \bigcup Q_i$,  the set cover problem is to identify the smallest subfamily of $\mathcal{Q}$ that still covers $Q$.

Now we will develop the path to reduce the set cover problem to the problem of finding a \emph{minimal generator} for some critical set in closure system without $D$-cycles (see section \ref{kbas} to recall definitions). This is equivalent to identifying a premise for one non-binary implication in an optimum basis. First, we illustrate it on a simple example.

\begin{exm}\label{B4double}
Let $Q=\{q_1,q_2,q_3,q_4\}$ be a finite set with the set cover $\mathcal{Q}=\{\{q_1\}, \{q_2\}, \{q_3\}, \{q_4\}, \{q_1, q_2\}\}$. Apparently, the solution to the set cover problem will be given by family $\mathcal{B}=\{\{q_3\}, \{q_4\}, \{q_1, q_2\}\}$.

We will construct a closure system without $D$-cycles on an extension $S$ of $Q$ so that $Q$ will be the minimal \emph{order} generator for some critical set. At the same time, $Q$ will not be a minimal generator, and the latter will be rather associated with the subfamily $\mathcal{B}$. 
 
The procedure will use a lattice theoretical construction to build $L=\op{Cl}(S,\phi)$, for a standard system $\<S,\phi\>$. 
For every subset $Y$ of family $\mathcal{Q}$ which is a not a singleton, add a new element $z(Y)$ to $Q$, also add another element $w$. Thus, with notation $z= z(\{q_1q_2\})$, we have in our case $S=\{q_1,q_2,q_3,q_4,z,w\}$. 

Start from $L_0=\mathbf{2}^Q$. Then use the \emph{doubling construction} introduced in A.~Day  \cite{D79} to double the element $b=\{q_1,q_2\}\in L_0$. This will replace $b$ by a $2$-element interval $[b,z]$ and extend lattice operations so that $b\vee x=z\vee x$, for every $x \not \leq b$, and $b\wedge x=z\wedge x$, for every $x \not \geq b$, $x \in L_0$. Let $L_1$ be the lattice after this doubling. Then $z$ is a join irreducible element in $L_1$; moreover, $z\geq q_1,q_2$. 

Finally, for $w$, choose any maximal proper subset in $Q$, say, $\{q_1,q_2,q_3\}$, and let $t\in L_1$ be the element inherited from $\{q_1,q_2,q_3\}\in L_0$. Apply the doubling construction to $t$ to obtain a $2$-element interval $[t,w]$, so a new lattice $L_2$ after this second doubling will have $w$ as a new join irreducible element. 

Let $\<S,\phi\>$ be a closure system with $\op{Cl}(S,\phi)=L_2$. It is known that every lattice obtained from a Boolean lattice by the series of doublings of intervals is a \emph{bounded} lattice, so that in particular, $\<S,\phi\>$ is without $D$-cycles, see A.~Day \cite{D79}.

It is straightforward to check that $C=\{q_1,q_2,q_3,q_4,z\}$ is a critical set in this closure system with $\phi(C)=S$.
Indeed, $z \in \phi(q_1,q_2,q_3)\subset \phi(C)$ and $\phi(C)\not = C$. Besides, $Q$ is a minimal \emph{order} generator for $C$.

In particular, the implication $Q\rto w$ is in the $E$-basis (and $K$-basis).
Nevertheless, one can find a generator $B=\{z,q_3,q_4\}$ of smaller size. Indeed, since $z\rto q_1, z\rto q_2$ must be in the binary part of any basis for $\<S,\phi\>$, we have $\sigma(B)=C$. 
Thus, an optimum basis for $\<S,\phi\>$ will have $B\rto w$, not $Q\rto w$.

Note that every element in $B$ corresponds to an element of the family $\mathcal{B}$ that solves the original set cover problem.
\end{exm}

This example should clarify the general procedure described in the next statement.
We note that the set cover problem has an effective solution when $\mathcal{Q}$ contains $Q$, because the minimal cover will be just a single set. Similarly, it has an effective solution, if $\mathcal{Q}$ does  not have $Q$, but contains the set $Q_1=Q\setminus {q}$, for one of $q\in Q$. In this case the solution is two-element family $Q_1,Q_2$, where $q\in Q_2$. These instances of set cover problem will be considered \emph{trivial}.

\begin{lem}\label{bounded}
Suppose $(Q,\mathcal{Q})$ is a non-trivial instance of the set cover problem. One can effectively find a closure system, whose closure lattice is bounded, such that a particular non-binary implication $U\rto V$ in its optimum basis translates into the solution of this set cover problem.
\end{lem}
\begin{proof}
Due to assumption, $\mathcal{Q}$ does not have $Q$ or any $Q\setminus \{q\}$. Suppose $Q=\{q_1,\dots, q_n\}$ and $\mathcal{Q}$ has $k$ elements of cardinality $>1$.

First, we make a slight modification to the instance of the set cover problem. Extend the family $\mathcal{Q}$ to 
$\mathcal{Q^*}=\mathcal{Q}\cup \{\{q_i\}: i\leq n\}$. A solution to the instance $(Q,\mathcal{Q^*})$ should be of cardinality at most as large as the solution to original instance $(Q,\mathcal{Q})$, simply because we might have more available covering subfamilies in second instance. On the other hand, if the solution to the second instance contains any one-element subsets which are not in $\mathcal{Q}$, we can replace them by subsets in $\mathcal{Q}$ that contain  those singletons, obtaining the solution at most as large as the solution to the second instance. 

Thus, we may replace original instance by the instance $(Q,\mathcal{Q^*})$. 

To build a finite bounded lattice $L$, start from $L_0=\mathbf{2}^Q$, then double every element $b_i\in 2^Q$ that corresponds to $Q_i\in \mathcal{Q}$, $|Q_i|>1$, $i\leq k$. This will add new join-irreducible elements $z_i$, $i\leq k$. It implies, in particular, that  $z_i\geq q$, for every $q \in Q_i$.

Also, choose an arbitrary element $t$ corresponding to some $Q\setminus \{q\} \in L_0$, and also double it, adding a new join-irreducible element $w$. The resulting lattice $L$ is the closure lattice of a standard closure system on its set of join-irreducible elements $S=\{q_1,\dots, q_n,z_1,\dots, z_k, w\}$. Due to the nature of the construction, $L$ is a bounded lattice, in particular, the standard closure system $\<S,\phi\>$ corresponding to $L$ is without $D$-cycles.  

It is straightforward to verify that $C=\{q_1,\dots, q_n, z_1,\dots, z_k\}$ is a (unique) critical set of this closure system  with $\phi(C)=S$. Hence, every optimum basis for $\<S,\phi\>$ should have an implication $U\rto w$, where $U\subseteq C$ is a minimal generator for $C$. 

Also, $Q$ is a unique $\ll$-minimal representation of $S$. Hence, every $W\subseteq S$ with $\phi(W)=S$ should satisfy $Q\ll W$. Since it should hold for $U$,  we should have for every $q \in Q$ some $u \in U$ such that $u\geq_\phi q$, or simply $u\geq q$ in $L$. 

Every element in $U$ can be interpreted as an element of family $\mathcal{Q^*}$: if $u=q_i$, then it is $\{q_i\}$, and if $u=z_i$, then it is $Q_i=\phi(z_i)\cap Q$. Thus, $U$ can be interpreted as a covering subfamily of $\mathcal{Q^*}$.

If we find a minimal generator $U$ for $C$, then it will serve as a solution to the instance $(Q,\mathcal{Q^*})$ of the set cover problem.
\end{proof}

\begin{cor}\label{NP1}
The problem of finding a basis $\Sigma$, for a closure system without $D$-cycles (and even closure systems with bounded closure lattice), which reaches minimum in $s_L(\Sigma^{nb})$, is NP-complete.
\end{cor}

\begin{proof}
Suppose, we can find a polynomial algorithm of finding such a basis for a closure systems without $D$-cycles, given its $E$-basis. Then the the set cover problem can be effectively solved as well.
Indeed, start from any instance $(Q,\mathcal{Q})$ of the set cover problem. If it is trivial, then it is solved effectively. If it is not trivial, then we can effectively write the $E$-basis of a bounded lattice constructed in Lemma \ref{bounded} from $(Q,\mathcal{Q})$. Since, according to assumption, this allows us to write effectively 
an optimum basis for this closure system, the solution to the given set cover problem can be recovered from one of its implications in polynomial time. This contradicts the NP-completeness of the set cover problem.
\end{proof}

Now we turn to second problem on our list.

\begin{lem}\label{bounded+}
Suppose $(Q,\mathcal{Q})$ is a non-trivial instance of the set cover problem. One can effectively find a closure system, whose closure lattice is bounded, such that a particular binary implication $w\rto B$ of its optimum basis translates into a solution of this set cover problem.
\end{lem}
\begin{proof}
As in the proof of Lemma \ref{bounded}, we assume that $Q=\{q_1,\dots, q_n\}$, $\mathcal{Q}=\{Q_i: i \in I\}$ is some set cover of $Q$, and $\mathcal{Q}$ has $k$ elements of cardinality $>1$. We also can replace the instance of set cover problem by $(Q,\mathcal{Q^*})$ with $\mathcal{Q^*}=\mathcal{Q}\cup \{\{q_i\}: i\leq n\}$.

We start from $L_0=\mathbf{2}^Q\cup\{w\}$, where $w > x$, for all $x \in \mathbf{2}^Q$. For each $Q_i\in \mathcal{Q}$, $|Q_i|>1$, double element $b_i=Q_i\in L_0$, $i\leq k$, adding new join-irreducible elements $z_i$.
The resulting bounded lattice $L$ is built on the set of join-irreducible elements $S=\{q_1, \dots q_n,z_1, \dots, z_k,w\}$. Let $\<S,\phi\>$ be a closure system with $L=\op{Cl}(S,\phi)$.

According to Theorem \ref{rs-bin}, every (right-side) optimum basis should have an implication $w_0\rto  B$, where $\phi(B)=S\setminus \{w_0\}=\phi(\{w_0\})\setminus \{w_0\}$ and $|B|$ is minimal among subsets of $S\setminus\{w_0\}$ with this property. Recall that $Q$ is a unique $\ll$-minimal representation of $S\setminus \{w_0\}$. Hence, $Q\ll B$. This means that for every $q \in Q$ there exists $b\in B$ such that $b\geq_\phi q$, or simply $b\geq q$ in $L$. Thus, $B$ can be interpreted as a set cover for $Q$, where each $b\in B$ is thought as a an element of covering family $\mathcal{Q^*}$. The rest of the argument is similar to the proof of Lemma \ref{bounded}.

\end{proof}

\begin{cor}\label{sRNP}
The problem of finding a basis $\Sigma$, for systems without $D$-cycles (and even in closure systems with bounded closure lattices), that reaches minimum in $s_R(\Sigma^b)$, is NP-complete. In particular, the problem of finding an optimum basis for such systems in NP-complete.
\end{cor}

\begin{proof}
Indeed, starting from an instance of the cover set problem, build a finite bounded lattice $L$ from Lemma \ref{bounded+}. If the (right side) optimum basis for $L$ can be effectively found, then taking a binary implication corresponding to the top element of $L$, we would get a solution to the set cover problem. The last sentence follows
from Corollaries \ref{lr} and \ref{Bopt}.
\end{proof}


\emph{Acknowledgments.} The first draft of the paper was written during the first author's visit to University of Hawai'i, supported by AWM-NSF Travel Mentor Grant. The warm and encouraging atmosphere of the Department of Mathematics of UofH is highly appreciated. The work on the paper was also inspired by the communication with 
V.~Duquenne and E.~Boros.

\end{document}